\documentclass{amsart}

        \usepackage{latexsym}
        \usepackage{amssymb}
        \usepackage{amsmath}
        \usepackage{amsfonts}
        \usepackage{amsthm}
        \usepackage[hypertexnames=false]{hyperref}
        \usepackage[all,knot,poly,2cell]{xy}
             \newdir{(}{{}*!/-5pt/@^{(}}
             \newdir{(x}{{}*!/-5pt/@_{(}}
             \UseAllTwocells
        \usepackage{xspace}
        \usepackage{mdframed}[2012/01/08]
        \usepackage{mathtools} 

        \usepackage{mathrsfs}   
        \usepackage[mathcal]{eucal}

        \theoremstyle{plain}
        \newtheorem{thm}{Theorem}[section]
        
        \newtheorem{lem}[thm]{Lemma}
        \newtheorem{prop}[thm]{Proposition}
        \newtheorem*{mainthm}{Main Theorem}

        \theoremstyle{definition}
        \newtheorem{defn}[thm]{Definition}
        \newtheorem{ex}[thm]{Example}
        \newmdtheoremenv{boxcnd}[thm]{Condition} 
 
        \theoremstyle{remark}
        \newtheorem{rem}[thm]{Remark}

        \newcommand{\spref}[1]{\href{http://stacks.math.columbia.edu/tag/#1}{#1}}

        \renewcommand{\bar}{\overline}
        \newcommand{\suchthat}{\,:\,}
        
        \newcommand{\itemref}[1]{\eqref{#1}}
        \newcommand{\opcit}[1][]{[\emph{op.\ cit.}{#1}]\xspace}
        \newcommand{\loccit}{[\emph{loc.\ cit.}]\xspace}

        \numberwithin{equation}{section}

        \newcommand{\Z}{\mathbb{Z}}
        
        \newcommand{\Q}{\mathbb{Q}}

        \newcommand{\Orb}{\mathscr{O}}
        \DeclareMathOperator{\spec}{Spec}
        \DeclareMathOperator{\supp}{Supp}
        \newcommand{\red}[1]{{#1}_{\mathrm{red}}}
        
        \newcommand{\QCOHB}{\mathbf{QCoh}}
        \newcommand{\QCOH}[2][]{\QCOHB^{#1}({#2})}

        \newcommand{\Et}{\mathrm{\acute{E}t}}  
        \newcommand{\fppf}{\mathrm{fppf}}

        \newcommand{\et}{\mathrm{\acute{e}t}} 

        \DeclareMathOperator{\Der}{Der}
        \DeclareMathOperator{\Exal}{Exal}   
        \newcommand{\MOD}[1]{\mathbf{Mod}({#1})}

        \DeclareMathOperator{\Def}{Def}
        \newcommand{\EXAL}{{\mathbf{Exal}}}
        \DeclareMathOperator{\Obs}{Obs}
        
        \newcommand{\DEF}{{\mathbf{Def}}}
        \newcommand{\extn}[1]{[{#1}]}

        \DeclareMathOperator{\Hom}{Hom}
        \DeclareMathOperator{\Ext}{Ext}
        \DeclareMathOperator{\Tor}{Tor}
        \DeclareMathOperator{\coker}{coker}
        \newcommand{\DCAT}{\mathsf{D}}

        \DeclareMathOperator{\Aut}{Aut}
        
        \newcommand{\tensor}{\otimes}
        \newcommand{\opp}{\circ}
        \newcommand{\FIB}[2]{{#1}({#2})}
                \newcommand{\AB}{\mathbf{Ab}}
                        
                \newcommand{\Sch}{\mathbf{Sch}}
                \newcommand{\SCH}[2][]{{\Sch}^{#1}/{#2}}

\newenvironment{mydescription}{%
   
   \begin{description}%
}{%
   \end{description}%
}

\newcommand{\fndefn}[1]{\emph{{#1}}} 

\newcommand{\trvret}[2]{{r}_{{#1},{#2}}}
\newcommand{\GI}{{GI}\xspace}
\newcommand{\GS}{{GS}\xspace}
\newcommand{\GB}{{GB}\xspace}
\newcommand{\CI}{{CI}\xspace}
\newcommand{\CS}{{CS}\xspace}
\newcommand{\CB}{{CB}\xspace}
\newcommand{\Homog}[1]{\mathbf{{#1}}} 
\newcommand{\HNIL}{\Homog{Nil}}
\newcommand{\HCL}{\Homog{Cl}}
\newcommand{\HrNIL}{\Homog{rNil}}
\newcommand{\HrCL}{\Homog{rCl}}
\newcommand{\HFIN}{\Homog{Fin}}
\newcommand{\HINT}{\Homog{Int}}
\newcommand{\HArt}{\Homog{Art^{fin}}}
\newcommand{\HArtINSEP}{\Homog{Art^{insep}}}
\newcommand{\HArtTriv}{\Homog{Art^{triv}}}
\newcommand{\HA}{\Homog{Aff}}
\newcommand{\HDVR}{\Homog{DVR}}
\newcommand{\Van}{\mathbb{V}}
\newcommand{\COHB}{\mathrm{Coh}}

\title{Artin's criteria for algebraicity revisited} 
\date{June 18, 2013}
\author{Jack Hall}
\address{Department of Mathematics\\KTH Royal Institute of
  Technology\\SE-100 44 Stockholm\\Sweden}
\email{jackhall@math.kth.se}
\author{David Rydh}
\address{Department of Mathematics\\KTH Royal Institute of
  Technology\\SE-100 44 Stockholm\\Sweden}
\thanks{This collaboration was supported by the G\"oran Gustafsson foundation.
The second author is also supported by the Swedish Research Council 2011-5599.}
\email{dary@math.kth.se}

\begin{document}
\begin{abstract} 
  Using notions of homogeneity we give new proofs of M.\ Artin's algebraicity criteria for functors and groupoids. Our methods give a more general result, unifying Artin's two theorems and clarifying their differences.
\end{abstract}
\maketitle
\section*{Introduction}
Classically, moduli spaces in algebraic geometry are constructed using either
projective methods or by forming suitable quotients. In his reshaping of the
foundations of algebraic geometry half a century ago, Grothendieck shifted
focus to the functor of points and the central question became whether certain
functors are representable. Early on, he developed formal geometry and
deformation theory, with the intent of using these as the main tools for
proving representability. Grothendieck's proof of the existence of Hilbert and
Picard schemes, however, is based on projective methods. It was not until ten
years later that Artin completed Grothendieck's vision in a series of landmark
papers. In particular, Artin vastly generalized Grothendieck's existence result
and showed that the Hilbert and Picard schemes exist---as algebraic spaces---in
great generality. It also became clear that the correct setting was that of
algebraic spaces---not schemes---and algebraic stacks.


In his two eminent papers~\cite{MR0260746,MR0399094}, M.\ Artin gave precise
criteria for algebraicity of functors and stacks. These criteria were later
clarified and simplified by B.\ Conrad and J.\ de Jong~\cite{MR1935511}, who
replaced Artin approximation with N\'eron--Popescu desingularization, by
H.\ Flenner \cite{MR638811} using $\Exal$, and the first
author~\cite{hallj_openness_coh} using coherent functors. The criterion
in~\cite{hallj_openness_coh} is very streamlined and elegant and suffices---to
the best knowledge of the authors---to deal with all present problems. It does
not, however, supersede Artin's criteria as these are weaker. Another conundrum
is the fact that Artin gives two different criteria---the
first~\cite[Thm.~5.3]{MR0260746} is for functors and the
second~\cite[Thm.~5.3]{MR0399094} is for stacks---but neither completely
generalizes the other.

The purpose of this paper is to use the ideas of Flenner and the first author
to give a new criterion that supersedes all present criteria. We also introduce
several new ideas that strengthen the criteria and simplify the proofs
of~\cite{MR0260746,MR0399094,MR638811}.
In positive characteristic, we also identify a subtle issue in Artin's
algebraicity criterion for stacks. With the techniques that we develop, this
problem is circumvented.
We now state our criterion for
algebraicity.

\begin{mainthm} Let $S$ be an excellent scheme. Then a
category $X$,
fibered in groupoids over the category of $S$-schemes, $\SCH{S}$, is an
algebraic stack, locally of finite presentation over $S$, if and only if it
satisfies the following conditions.
\begin{enumerate}
  \item \label{mainthm:item:fppf-stack}
    $X$ is a stack over $(\SCH{S})_{\fppf}$.
  \item \label{mainthm:item:lp}
    $X$ is limit preserving (Definition~\ref{def:limit-pres}).
  \item \label{mainthm:item:homog}
    $X$ is $\HArtTriv$-homogeneous.
  \item \label{mainthm:item:eff}
    $X$ is effective (Definition~\ref{def:effective}).
\renewcommand{\theenumi}{5\alph{enumi}} 
\setcounter{enumi}{0}
  \item \label{mainthm:item:bdd}
    Automorphisms and deformations are bounded
    (Conditions~\ref{cnd:bdd_aut} and~\ref{cnd:bdd_def}).
  \item \label{mainthm:item:cons}
    Automorphisms, deformations and obstructions are constructible
    (Condition~\ref{cnd:cons_aut+def+obs}).
  \item \label{mainthm:item:Zar_loc}
    Automorphisms, deformations and obstructions are Zariski-local
    (Condition~\ref{cnd:Zar_loc_aut+def+obs});
    or $S$ is Jacobson; or $X$ is $\HDVR$-homogeneous
    (Definition~\ref{defn:DVR-homogeneity}).
\end{enumerate}
Condition~\ref{cnd:cons_obs} (resp.\ \ref{cnd:Zar_loc_obs}) on obstructions can
be replaced with either Condition~\ref{cnd:cons_obs:n-step} or
\ref{cnd:cons_obs:artin} (resp.\ either Condition~\ref{cnd:Zar_loc_obs:n-step},
or \ref{cnd:Zar_loc_obs:artin}). Finally, we may replace
\itemref{mainthm:item:fppf-stack} and \itemref{mainthm:item:homog} with
\begin{enumerate}\renewcommand{\theenumi}{$\arabic{enumi}'$}
  \item \label{mainthm:item:et-stack}
    $X$ is a stack over $(\SCH{S})_{\Et}$.
\addtocounter{enumi}{1}
  \item \label{mainthm:item:insep-homog}
    $X$ is $\HArtINSEP$-homogeneous.
\end{enumerate}
If every residue field of $S$ is perfect, e.g., if $S$ is a $\Q$-scheme or
of finite type over $\spec(\Z)$, then \itemref{mainthm:item:homog}
and \itemref{mainthm:item:insep-homog} are equivalent.
\end{mainthm}

The $\HArtTriv$-homogeneity (resp.\ $\HArtINSEP$-homogeneity) condition is the
following Schlessinger--Rim
condition: for any diagram of local artinian $S$-schemes of finite type
$[\spec B \leftarrow \spec A \hookrightarrow \spec A']$, where 
$A'\twoheadrightarrow A$ is surjective and the residue field extension
$B/\mathfrak{m}_B\to A/\mathfrak{m}_A$ is trivial (resp.\ purely
inseparable), the natural functor
\[
\FIB{X}{\spec (A'\times_A B)} \to \FIB{X}{\spec A'}
\times_{\FIB{X}{\spec A}} \FIB{X}{\spec B}
\]
is an equivalence of categories.

The perhaps most striking difference to Artin's conditions is that our
homogeneity condition~\itemref{mainthm:item:homog} only involves local artinian
schemes and that we do not need any conditions on \'etale localization of
deformation and obstruction theories. If $S$ is Jacobson, e.g., of finite type
over a field, then we do not even need compatibility with Zariski localization.
There is also no condition on compatibility with completions for
automorphisms and deformations. We will
do a detailed comparison between our conditions and other versions of
Artin's conditions in Section~\ref{sec:comparison}.

All existing algebraicity proofs, including ours, consist of the following
four steps:
{\renewcommand{\theenumi}{\roman{enumi}}
\begin{enumerate}
\item existence of formally versal deformations;\label{step:formal}
\item algebraization of formally versal deformations;\label{step:algebraization}
\item openness of formal versality; and\label{step:openness-of-fv}
\item formal versality implies formal smoothness.\label{step:fv-fs}
\end{enumerate}}
Step~\itemref{step:formal} was eloquently dealt with by
Schlessinger~\cite[Thm.\ 2.11]{MR0217093} for functors and
Rim~\cite[Exp.~VI]{SGA7} for groupoids. This step uses
conditions~\itemref{mainthm:item:homog}
and~\itemref{mainthm:item:bdd} ($\HArtTriv$-homogeneity and boundedness of
tangent spaces). Step~\itemref{step:algebraization} begins
with the effectivization of formally versal deformations using
condition~\itemref{mainthm:item:eff}. One may then algebraize this
family using either Artin's results~\cite{MR0268188,MR0260746} or B.\
Conrad and J.\ de Jong's result~\cite{MR1935511}. In the latter approach,
Artin approximation is replaced with N\'eron--Popescu
desingularization and $S$ is only required to be excellent. This step requires 
condition~\itemref{mainthm:item:lp}.

The last two steps are more subtle and it is here that
\cite{MR0260746,MR0399094,MR638811,starr-2006,hallj_openness_coh} and our
present treatment diverges---both when it comes to the criteria themselves and
the techniques employed. We begin with discussing step~\itemref{step:fv-fs}.

It is readily seen that our criterion is weaker than Artin's two
criteria~\cite{MR0260746,MR0399094} except that, in positive characteristic, we
need $X$ to be a stack in the fppf topology, or otherwise
strengthen~\itemref{mainthm:item:homog}. This is similar
to~\cite[Thm.~5.3]{MR0260746} where the functor is assumed to be an
fppf-sheaf. In \loccit, Artin uses the fppf sheaf condition and a clever
descent argument to deduce that
formally universal deformations are formally
\'etale~\cite[pp.~50--52]{MR0260746}, settling step~\itemref{step:fv-fs} for
functors. This argument relies on the existence of universal deformations
and thus does not extend to stacks with infinite or non-reduced stabilizers.

In his second paper~\cite{MR0399094}, Artin only assumes that the
groupoid is an \'etale stack. His proof of step~\itemref{step:fv-fs} for
groupoids~\cite[Prop.\ 4.2]{MR0399094}, however, does not treat inseparable
extensions. We do not understand how this problem can be overcome without
strengthening the criteria and assuming that either~\itemref{mainthm:item:fppf-stack}
the groupoid is a stack
in the fppf topology or~\itemref{mainthm:item:insep-homog} requiring homogeneity
for inseparable extensions.
Flenner does not discuss formal smoothness,
and in~\cite{hallj_openness_coh} formal smoothness is obtained by strengthening
the homogeneity condition~\itemref{mainthm:item:homog}.

With a completely different and
simple argument, we show that formal versality and formal smoothness are equivalent.
The idea is that with \emph{homogeneity}, rather than \emph{semi-homogeneity},
we can use the stack condition~\itemref{mainthm:item:fppf-stack} to obtain
homogeneity for artinian rings with arbitrary residue field extensions
(Lemma~\ref{lem:HArt}). This immediately implies that formal versality and
formal smoothness are equivalent (Lemma~\ref{lem:smooth}) so we accomplish
step~\itemref{step:fv-fs} without using obstruction theories.

Finally, Step~\itemref{step:openness-of-fv} uses constructibility, boundedness,
and Zariski localization of deformations and obstruction theories
(Theorem~\ref{thm:fv_art_flenner}). In our treatment, localization is only
required when passing to non-closed points of finite type. Such
points only exist when $S$ is not Jacobson, e.g., if $S$ is the spectrum of a
discrete valuation ring. Our proof is very similar to Flenner's proof. It may
appear that Flenner does not need Zariski localization in his criterion, but
this is due to the fact that his conditions are expressed in terms of
deformation and
obstruction \emph{sheaves}.

As in Flenner's proof, openness of versality becomes a matter of simple
algebra. It comes down to a criterion for the openness of the \emph{vanishing
  locus} of half-exact functors (Theorem~\ref{thm:flenner_vl}) that easily
follows from the Ogus--Bergman Nakayama Lemma for half-exact functors
(Theorem~\ref{thm:nakayama}). Flenner proves a stronger statement that implies
the Ogus--Bergman result (Remark~\ref{rem:nakayama-Flenner}).

At first, it seems that we need more than $\HArtTriv$-homogeneity to even make
sense of conditions~\itemref{mainthm:item:bdd}--\itemref{mainthm:item:Zar_loc}.
This will turn out to not be the case. Using
steps~\itemref{step:algebraization} and~\itemref{step:fv-fs}, we prove that
conditions~\itemref{mainthm:item:fppf-stack}--\itemref{mainthm:item:eff}
guarantee that we have homogeneity for arbitrary integral morphisms
(Lemma~\ref{lem:prorep_finhomg_affhomg}). It follows that $\Aut_{X/S}(T,-)$,
$\Def_{X/S}(T,-)$ and $\Obs_{X/S}(T,-)$ are additive functors.

\subsection*{Outline}
In Section~\ref{sec:homogeneity} we recall the notions of homogeneity, limit
preservation and
extensions from~\cite{hallj_openness_coh}. We also introduce homogeneity that
only involves artinian rings and show that residue field extensions are
harmless for stacks in the fppf topology. In Section~\ref{sec:fv_fs} we then
relate formal
versality, formal smoothness and vanishing of $\Exal$.

In Section~\ref{sec:vl} we study additive functors and their vanishing
loci. This is applied in Section~\ref{sec:op_fv} where we give conditions on
$\Exal$ that assure that the
locus of formal versality is open. The results are then assembled in
Theorem~\ref{thm:fv_art_flenner}.

In Section~\ref{sec:aut_def_obs} we repeat the definitions of automorphisms,
deformations and
minimal obstruction theories from~\cite{hallj_openness_coh}. In
Section~\ref{sec:rel_conds}, we
give conditions on $\Aut$, $\Def$ and $\Obs$ that imply the corresponding
conditions on $\Exal$ needed in Theorem~\ref{thm:fv_art_flenner}. In
Section~\ref{sec:obs-theories}
we introduce $n$-step obstruction theories.
In Section~\ref{sec:no-obs-theory} we formulate the conditions on obstructions
without using
linear obstruction theories, as in~\cite{MR0260746}.
Finally, in Section~\ref{sec:mainthm} we prove the Main Theorem. Comparisons
with other criteria are given in Section~\ref{sec:comparison}.

\subsection*{Notation}
We follow standard conventions and notation. In particular, we adhere to the
notation of~\cite{hallj_openness_coh}.
Recall that if $T$ is a scheme, then a point $t\in |T|$ is of \emph{finite type}
if
$\spec(\kappa(t))\to T$ is of finite type. Points of finite type are locally
closed. A point of a Jacobson scheme is of finite type if and only if it is
closed. If $f\colon X\to Y$ is of finite type and $x\in |X|$ is of finite type, then
$f(x)\in |Y|$ is of finite type.

\subsection*{Acknowledgment}
We would like to thank M.\ Artin for encouraging comments and
L.\ Moret--Bailly for answering a question on MathOverflow about Jacobson
schemes.

\section{Homogeneity, limit preservation, and extensions}\label{sec:homogeneity}
In this section, we review the concept of homogeneity---a
generalization of Schlessinger's Conditions that we attribute to
J.\ Wise \cite[\S2]{2011arXiv1111.4200W}---in the formalism of 
\cite[\S\S1--2]{hallj_openness_coh}. We will also briefly discuss limit
preservation and extensions.

Fix a scheme $S$. An 
$S$-\fndefn{groupoid} is a category $X$, together with a functor $a_X
\colon X \to \SCH{S}$ that is fibered in groupoids. A $1$-morphism of
$S$-groupoids $\Phi \colon (Y,a_Y) \to 
(Z,a_Z)$ is a functor between categories $Y$ and $Z$ that
commutes strictly over $\SCH{S}$. We will typically refer to an
$S$-groupoid $(X,a_X)$ as ``$X$''.

An $X$-\fndefn{scheme} is a pair $(T,\sigma_T)$, where $T$ is an 
$S$-scheme and $\sigma_T \colon \SCH{T} \to X$ is a $1$-morphism of
$S$-groupoids. A morphism of $X$-schemes $U \to V$ is a morphism of
$S$-schemes $f \colon {U} \to {V}$ (which canonically determines a
$1$-morphism of $S$-groupoids $\SCH{f} \colon \SCH{U} \to \SCH{V}$)
together with a $2$-morphism $\alpha \colon \sigma_U  \Rightarrow
\sigma_V\circ \SCH{f}$. The collection of all 
$X$-schemes forms 
a $1$-category, which we denote as $\SCH{X}$. It is readily seen that
$\SCH{X}$ is an $S$-groupoid and that there is a natural equivalence
of $S$-groupoids $\SCH{X}\to X$.
For a $1$-morphism of $S$-groupoids $\Phi \colon Y \to Z$ there is an
induced functor $\SCH{\Phi} \colon \SCH{Y} \to \SCH{Z}$.

We will be interested in the following classes of morphisms of $S$-schemes:
\begin{itemize}
\item[$\HNIL$] -- locally nilpotent closed immersions,
\item[$\HCL$] -- closed immersions,
\item[$\HrNIL$] -- morphisms $X\to Y$ such that there exists $(X_0\to X)\in \HNIL$ with the composition $(X_0\to X\to Y)\in \HNIL$,
\item[$\HrCL$] -- morphisms $X\to Y$ such that there exists $(X_0\to X)\in \HNIL$ with the composition $(X_0\to X\to Y)\in \HCL$,
\item[$\HArt$] -- morphisms between local artinian schemes of finite type over
  $S$,
\item[$\HArtINSEP$] -- $\HArt$-morphisms with purely inseparable residue
  field extensions,
\item[$\HArtTriv$] -- $\HArt$-morphisms with trivial residue field
  extensions,
\item[$\HFIN$] -- finite morphisms,
\item[$\HINT$] -- integral morphisms,
\item[$\HA$] -- affine morphisms.
\end{itemize}
We certainly have a containment
of classes of morphisms of $S$-schemes:
\[
\xymatrix@-1.5pc{ \HNIL\ar@{}[r]|-{\subset}\ar@{}[d]|-{\cap}
  & \HCL \ar@{}[d]|-{\cap} & & \\
  \HrNIL\ar@{}[r]|-{\subset} & \HrCL\ar@{}[r]|-{\subset}
  & \HINT\ar@{}[r]|-{\subset} & \HA. \\
  \HArtTriv \ar@{}[u]|-{\cup}\ar@{}[r]|-{\subset}
  & \HArtINSEP \ar@{}[r]|-{\subset} & \HArt \ar@{}[u]|-{\cup} } 
\]
Note that for a morphism $X\to Y$ of locally noetherian $S$-schemes, the
properties $\HrNIL$ and $\HrCL$ simply mean that $\red{X}\to Y$ is $\HNIL$
and $\HCL$ respectively.

Let $P\subseteq \HA$ be a class of morphisms. In
\cite[\S1]{hallj_openness_coh} the notion of a $P$-homogeneous
$1$-morphism of $S$-groupoids $\Phi \colon Y \to Z$ was defined. We say
that an $S$-groupoid is $P$-homogeneous if its structure $1$-morphism is. We will
not recall the definition in full (it is somewhat lengthy), but we
will give an explicit description in Lemma~\ref{lem:hom_approx} for
$S$-groupoids that are stacks in the Zariski topology. We will also show that
for \emph{limit preserving} Zariski stacks, it is enough to verify
$P$-homogeneity for $S$-schemes of finite type.

\begin{defn}\label{def:limit-pres}
Let $X$ be an $S$-groupoid that is a Zariski stack. We say that $X$ is \fndefn{limit
  preserving} if for any inverse system of
affine $S$-schemes $\{\spec A_j\}_{j\in J}$, with limit $\spec A$, the natural
functor:
\[
\varinjlim_j \FIB{X}{\spec A_j} \to \FIB{X}{\spec A}
\]
is an equivalence of categories \cite[\S1]{MR0399094}.
\end{defn}
The definition
just given also agrees with the definition in 
\cite[\S3]{hallj_openness_coh}. When $X$ is an algebraic stack, then $X$
is limit preserving if and only if $X\to S$ is locally of finite presentation
\cite[Prop.~4.15]{MR1771927}.

\begin{lem}\label{lem:hom_approx}
Let $S$ be a scheme. Consider a class of morphisms $P \subset \HA$ that is local
for the Zariski topology. Let $X$ be an $S$-groupoid that is a stack for the
Zariski topology. Then the following conditions are equivalent. 
\begin{enumerate}
\item $X$ is $P$-homogeneous.
\item\label{lem:hom_approx:item:aff} For any diagram of affine schemes
  $[\spec B \leftarrow \spec A \xrightarrow{{i}} \spec A']$, where ${i}$
  is a nilpotent closed immersion and $\spec A \to \spec B$ is $P$, the
  natural functor:
\[
\FIB{X}{\spec (A'\times_A B)} \to \FIB{X}{\spec A'}
\times_{\FIB{X}{\spec A}} \FIB{X}{\spec B}
\]
is an equivalence of categories. 
\end{enumerate}
If, in addition, $X$ is limit preserving, and $P\in
\{\HNIL,\HCL,\HrNIL,\HrCL,\HINT,\HA\}$, then in
\itemref{lem:hom_approx:item:aff} it suffices to take $\spec A$, $\spec A'$,
and $\spec B$ to be locally of finite presentation over $S$. In particular,
$\HINT$-homogeneity is equivalent to $\HFIN$-homogeneity and,
if $S$ is locally noetherian, then $\HrCL$-homogeneity is equivalent to the
condition $(\mathrm{S}1')$ of \cite[2.3]{MR0399094}.
\end{lem}
\begin{proof}
The first part follows from the definitions. To see the second part, assume
that $X$ is limit preserving and that $P\in
\{\HNIL,\HCL,\HrNIL,\HrCL,\HINT,\HA\}$. As $X$ is a Zariski stack we may assume
that $S=\spec(R)$ is affine. Let $[\spec B \leftarrow \spec A
  \xrightarrow{{i}} \spec A']$ be a diagram as in
\itemref{lem:hom_approx:item:aff} and let $B'=A'\times_A B$. Then, by
Proposition~\ref{prop:pushout-approx}, it can be written as an inverse limit of
diagrams $[\spec B_\lambda \leftarrow \spec
  A_\lambda \xrightarrow{{i}} \spec A'_\lambda]$ of finite presentation over
$S$ and, furthermore, $\spec B' $ is the inverse limit of
$\spec B'_\lambda$, where $B'_\lambda=A'_\lambda\times_{A_\lambda} B_\lambda$.
The result then follows from our assumption that $X$ is limit preserving.
\end{proof}

By \cite[Prop.\
2.1]{2011arXiv1111.4200W}, any algebraic stack is $\HA$-homogeneous.
It is easily verified, as is done in~\loccit, that if the stack is not
necessarily algebraic but has representable diagonal, then the functor above is
at least fully faithful. Moreover,
$\HrCL$-homogeneity is equivalent to Artin's \emph{semi}-homogeneity
condition \cite[2.2(S1a)]{MR0399094} for $X$, its diagonal $\Delta_X$,
and its double diagonal $\Delta_{\Delta_X}$. 

The main computational tool that $P$-homogeneity brings is \cite[Lem.\
1.4]{hallj_openness_coh}, which we now recall. 
\begin{lem}\label{lem:homog_pushouts_int} 
  Let $S$ be a scheme and let  $P\subset \HA$ be a class of morphisms.
  Let $X$ be a $P$-homogeneous $S$-groupoid. Consider a diagram of
  $X$-schemes $[V \xleftarrow{p} T \xrightarrow{{i}} T']$, where
  ${i}$ is a locally nilpotent closed immersion and $p$ is
  $P$. Then there exists a cocartesian diagram in the category of
  $X$-schemes: 
  \[
  \xymatrix@-0.8pc{T \ar@{(->}[r]^{i} \ar[d]_p & T'  \ar[d]^{p'} \\ V
      \ar@{(->}[r]^{{i}'}  & V'. }
  \]
  This diagram is also cocartesian in the category of
  $S$-schemes, the morphism ${i}'$ is a locally nilpotent closed
  immersion, $p'$ is affine, and the induced homomorphism of sheaves: 
  \[
  \Orb_{V'} \to {i}'_* \Orb_V \times_{p'_*{i}_*\Orb_T} p'_*\Orb_{T'}
  \]
  is an isomorphism. Moreover if $P\in \{\HNIL,\HCL,\HrNIL,\HrCL,\HFIN,\HINT,
  \HArt,\HArtINSEP,\HArtTriv\}$, then $p'$ is $P$.  
\end{lem}
\begin{proof}
Everything except the last claim is \cite[Lem.\ 1.4]{hallj_openness_coh}.
The last claim is trivial except for $P\in
\{\HNIL,\HCL,\HFIN,\HINT\}$. In
these cases, however, it is well-known---see e.g.~\cite[5.6 (3)]{MR2044495}.
\end{proof}

\begin{rem}
Let $S$ be a noetherian scheme. If $[\spec B \leftarrow \spec A
  \xrightarrow{{i}} \spec A']$ is a diagram of schemes of finite type over a
scheme $S$ such that $\spec A\to \spec B$ is integral (or equivalently finite)
and $\spec A \to \spec A'$ is a locally nilpotent immersion, then $\spec
(B\times_A A')$ is of finite type over $S$. This follows from the fact that
$B\times_A A'\subset B\times A'$ is an integral
extension~\cite[Prop.~7.8]{MR0242802}.  On the other hand, if $\spec A\to \spec
B$ is only affine, then $\spec (B\times_A A')$ is typically not of finite type
over $S$. For example, if $B=k[x]$, $A=k[x,x^{-1}]$ and $A'=k[x,x^{-1},y]/y^2$,
then $B'=B\otimes_A A'=k[x,y,yx^{-1},yx^{-2},\dots]/(y,yx^{-1},\dots)^2$ which
is not of finite type over $S=\spec(k)$.
\end{rem}

Homogeneity supplies an $S$-groupoid with a quantity of linear data,
which we now recall from \cite[\S2]{hallj_openness_coh}. An
$X$-\fndefn{extension} is     
a square zero closed immersion of $X$-schemes ${i} \colon T
\hookrightarrow T'$. 
The collection of $X$-extensions forms a category,
which we denote as $\EXAL_X$. There is a natural functor $\EXAL_X \to
\SCH{X}$ that takes $({i} \colon T\hookrightarrow T')$ to $T$.

We denote by $\EXAL_X(T)$ the fiber
of the category $\EXAL_X$ over the $X$-scheme $T$---we call these the
$X$-extensions of $T$. There is a natural functor 
\[
\EXAL_X(T)^\opp \to \QCOH{T},\quad ({i} \colon T \hookrightarrow T')
\mapsto \ker({i}^{-1}\Orb_{T'} \to \Orb_T).
\]
We denote by $\EXAL_X(T,I)$ the fiber category of $\EXAL_X(T)$ over 
the quasi-coherent $\Orb_T$-module $I$---we refer to these as the 
$X$-extensions of $T$ by $I$. 

Let $W$ be a scheme and let $J$ be a quasi-coherent $\Orb_W$-module. We let $W\extn{J}$
denote the $W$-scheme $\underline{\spec}_W (\Orb_W \extn{J})$ with structure
morphism $\trvret{W}{J} \colon W\extn{J} \to W$. If $W$ is an $X$-scheme, we
consider $W\extn{J}$ as an $X$-scheme via $\trvret{W}{J}$. The $X$-extension
$W\hookrightarrow W\extn{J}$ is thus \emph{trivial} in the sense that it admits
an $X$-retraction.

By \cite[Prop.\
2.3]{hallj_openness_coh}, if the $S$-groupoid $X$ is
$\HNIL$-homogeneous, then the groupoid $\EXAL_X(T,I)$ is a Picard
category.  Denote the set of isomorphism classes of the category
$\EXAL_X(T,I)$ by $\Exal_X(T,I)$. Thus, we have additive functors 
\begin{align*}
  \Der_X(T,-) &\colon \QCOH{T} \to \AB,\quad I \mapsto
  \Aut_{\EXAL_X(T,I)}(T\extn{I})\\
  \Exal_X(T,-) &\colon \QCOH{T} \to \AB,\quad I \mapsto \Exal_X(T,I). 
\end{align*}
We now record here the following easy
consequences of 
\cite[2.2--2.5 \& 3.4]{hallj_openness_coh}. 
\begin{lem}\label{lem:der_exal_props_record}
  Let $S$ be a scheme, let $X$ be an $S$-groupoid, and let $T$ be an
  $X$-scheme.
  \begin{enumerate}
  \item\label{lem:der_exal_props_record:item:trv} Let $I$ be a quasi-coherent
    $\Orb_T$-module. Then $\Exal_X(T,I) = 0$ if and only if every
    $X$-extension ${i} \colon T\hookrightarrow T'$ of $T$ by $I$ admits an $X$-retraction.
  \item\label{lem:der_exal_props_record:item:he} If $X$ is $\HrNIL$-homogeneous,
    then the functor $M\mapsto \Exal_X(T,M)$ is half-exact. 
  \item\label{lem:der_exal_props_record:item:lp} Suppose that $X$ is
    $\HNIL$-homogeneous and limit preserving. If $T$ is locally of finite
    presentation over $S$,
    then the functor $M \mapsto \Exal_X(T,M)$ preserves direct limits.
  \item\label{lem:der_exal_props_record:item:et}
    Let $p \colon U \to T$ be an \emph{affine} \'etale morphism and let $N$
    be a quasi-coherent $\Orb_U$-module. Then there is a natural functor $\psi \colon \EXAL_X(T,p_*N) \to
    \EXAL_X(U,N)$. If $({i} \colon T \hookrightarrow T')\in \EXAL_X(T,p_*N)$
    with image $({j} \colon U \hookrightarrow U')\in\EXAL_X(U,N)$, then there is
    a cartesian diagram of $X$-schemes
    \[
    \xymatrix@-0.8pc{U \ar@{(->}[r]^{{j}} \ar[d]_p & U'  \ar[d]^{p'} \\
      T \ar@{(->}[r]^{{i}} & T',}
    \]
    which is cocartesian as a diagram of $S$-schemes.
    If $X$ is $\HA$-homogeneous, then $\psi$ is an equivalence.
  \end{enumerate}
\end{lem}
Finally, we give conditions that imply $\HArt$-homogeneity.
\begin{lem}\label{lem:HArt}
Let $S$ be a scheme and let $X$ be an $S$-groupoid that is $\HArtTriv$-homogeneous.
Assume that one of the following conditions are satisfied.
\begin{enumerate}
\item $X$ is a stack in the fppf topology.\label{lem:HArt:fppf}
\item $X$ is a stack in the \'etale topology and
  $\HArtINSEP$-homogeneous.\label{lem:HArt:et+insephom}
\item $S$ is a $\Q$-scheme and $X$ is a stack in the \'etale
  topology.\label{lem:HArt:et+char0}
\end{enumerate}
Then $X$ is $\HArt$-homogeneous.
\end{lem}
\begin{proof}
We begin by noting that trivially \itemref{lem:HArt:et+char0} implies
\itemref{lem:HArt:et+insephom}. Next, let $[\spec B \leftarrow \spec A
  \hookrightarrow \spec A']$ be a diagram of local artinian $S$-schemes, with
$A' \twoheadrightarrow A$ a surjection of rings with nilpotent kernel, and
$B\to A$ finite so that $\spec A\to \spec B$ belongs to $\HArt$.
Let
$\spec B'=\spec(A'\times_A B)$ be the pushout of this diagram in the category
of $S$-schemes. We have to prove that the functor
\[
\varphi \colon \FIB{X}{\spec (A'\times_A B)} \to \FIB{X}{\spec A'}
\times_{\FIB{X}{\spec A}} \FIB{X}{\spec B}
\]
is an equivalence. Assume that $X$ is $\HArtTriv$-homogeneous
(resp.\ $\HArtINSEP$-homogeneous). We first show that $\varphi$ is an
equivalence when $A$, $A'$ and $B$ are not necessarily local but the residue
field extensions of $\spec(A)\to \spec(B)$ are trivial (resp.\ purely
inseparable). As $\spec B\hookrightarrow \spec B'$ is bijective, and $X$ is a
Zariski
stack, we can work locally on $\spec B'$ and assume that $\spec B'$ is
local. Then $\spec B$ is also local and if we let $A=\prod_{i=1}^n A_i$ and
$A'=\prod_{i=1}^n A'_i$ be decompositions such that $A'\twoheadrightarrow A_i$
factors through $A'_i$, then $B'=(A'_1\times_{A_1} B)\times_B
(A'_2\times_{A_2} B) \times_B \dots \times_B (A'_n\times_{A_n} B)$ is an
iterated fiber product of local artinian rings. The equivalence of $\varphi$ in
the non-local case thus follows from the local case.

If $X$ is a stack in the fppf (resp.\ \'etale) topology, then the equivalence of
$\varphi$ is a local question in the fppf (resp.\ \'etale) topology on $B'$
since fiber products of rings commute with flat base change. As $\spec
B\hookrightarrow \spec B'$ is a nilpotent closed immersion, the scheme $\spec
B'$ is local artinian and the residue fields of $B$ and $B'$ coincide. Choose a
finite (resp.\ finite separable) field extension $K/k_B$ such that the residue
fields of $k_A\otimes_{k_B} K$ are trivial (resp.\ purely inseparable)
extensions of $K$. There is then a local artinian ring $\widetilde{B}'$ and a
finite flat (resp.\ finite \'etale) extension $B'\hookrightarrow
\widetilde{B}'$ with $k_{\widetilde{B}'}=K$. Let $\widetilde{A}=A\otimes_{B'}
\widetilde{B}'$, $\widetilde{A}'=A'\otimes_{B'} \widetilde{B}'$ and
$\widetilde{B}=B\otimes_{B'}
\widetilde{B}'$. Then $\widetilde{A}$, $\widetilde{A}'$, $\widetilde{B}$
are artinian rings such that all residue fields equal $K$ (resp.\ are purely
inseparable extensions of $K$). Thus, equivalence of $\varphi$ follows from
the case treated above.
\end{proof}

\section{Formal versality and formal smoothness}\label{sec:fv_fs}
In this section we address a subtle point about the relationship
between formal versality and formal smoothness. To be precise, we
desire sufficient conditions for a family,
formally versal at all \emph{closed} points, to be formally \emph{smooth}.
In the algebraicity criterion for functors \cite[Thm.\ 5.3]{MR0260746} a
precise statement in this form is not present, but is addressed in
\opcit[, Lem.\ 5.4]. In the algebraicity criterion for 
groupoids \cite[Thm.\ 
5.3]{MR0399094} the relevant result is precisely stated in \opcit[, 
Prop.\ 4.2]. We do not, however, understand the proof.

In the notation of \loccit, to verify formal smoothness, the residue fields of
$A$ are not fixed. But the proof of \loccit\ relies on \opcit[, Thm.~3.3],
which requires that the residue field of $A$ is equal to the residue field of
$R$. If the residue field extension is separable, then it is possible to
conclude using \opcit[, Prop.~4.3], which uses \'etale localization of
obstruction theories.  We do not know how to complete the argument if
the residue field extension is inseparable. The essential problem is the
verification that formal versality is smooth-local.

We also wish to
point out that, in \loccit, the techniques of Artin approximation are
used via \opcit[, Prop.\ 3.3].
In this section we demonstrate that
excellence (or related) assumptions are irrelevant with our
formulation.

We begin this section with recalling, and refining, some results of
\cite[\S4]{hallj_openness_coh}. 
\begin{defn}
  Let $S$ be a scheme, let $X$ be an $S$-groupoid, and let $T$ be an $X$-scheme. Consider
  the following lifting problem: given a square zero closed immersion
  of $X$-schemes $Z_0 \hookrightarrow Z$ fitting into a
  commutative diagram of $X$-schemes:
  \[
  \xymatrix{ Z_0 \ar@{(x->}[d] \ar[r]^g  & T \ar[d]  \\ Z
    \ar[r] \ar@{-->}[ur]   & X.}
  \]
  The $X$-scheme $T$ is: 
  \begin{mydescription}
  \item[\fndefn{formally smooth}] if the lifting problem can always
    be solved \'etale-locally on $Z$;
  \item[\fndefn{formally smooth at $t\in |T|$}] if the
    lifting problem 
    can always be solved whenever the $X$-scheme $Z$ is local
    artinian, with closed point $z$, such that $g(z) = t$, and the
    field extension $\kappa(t) \subset \kappa(z)$ is finite;
  \item[\fndefn{formally versal at $t \in |T|$}] if the lifting
    problem can 
    always be solved whenever the $X$-scheme $Z$ is local artinian,
    with closed point $z$, such that $g(z) = t$ and $\kappa(t) \cong
    \kappa(z)$.
  \end{mydescription}
\end{defn}
We certainly have the following implications:
\begin{align*}
\text{formally smooth} &\Rightarrow \text{formally smooth at all $t\in
  |T|$}\\
    &\Rightarrow \text{formally versal at all $t\in |T|$}.
\end{align*}
It is readily observed that formal smoothness is smooth-local on the
source. Without stronger assumptions, it is not obvious to the
authors that formal versality is smooth-local on the source. Similarly,  
formal smoothness at $t$ and formal versality at $t$ are not
obviously equivalent. We will see, however, that these
subtleties vanish whenever the $S$-groupoid is $\HArt$-homogeneous.
For formal versality and formal smoothness at a point, it is sufficient
that liftings exist when
$\kappa(z) \cong g_*\ker (\Orb_{Z}\to \Orb_{Z_0})$.

\begin{lem}\label{lem:fsmooth_pt+repr}
  Let $S$ be a locally noetherian scheme and let $X$ be a limit preserving $S$-groupoid.
  Let $T$ be an $X$-scheme that is locally of
  finite type over $S$ and let $t\in |T|$ be a point such that:
  \begin{enumerate}
  \item $T$ is formally smooth at $t\in |T|$ as an $X$-scheme;
  \item the morphism $T \to X$ is representable by algebraic spaces.
  \end{enumerate}
  Let $W$ be an $X$-scheme. Then the morphism $T\times_X W\to W$ is smooth in a
  neighborhood of every point over $t$. In particular, if $T$ is formally
  smooth at every point of \emph{finite type}, then $T\to X$ is formally smooth.
\end{lem}
\begin{proof}
  By a standard limit argument we can assume that $W\to S$ is of finite type.
  It is then enough to verify that $T\times_X W\to W$ is smooth at closed points
  in the fiber of $t$ and this follows from~\cite[\textbf{IV}.17.14.2]{EGA}.
  The last
  statement follows from the fact that any closed point of $T\times_X W$ maps
  to a point of finite type of $T$.
\end{proof}
There is a tight connection between formal smoothness
(resp.\ formal 
versality) and $X$-extensions in the \emph{affine} setting. Most of
the next result was proved in \cite[Lem.\ 4.3]{hallj_openness_coh},
which utilized arguments similar to those of \cite[Satz~3.2]{MR638811}.
\begin{lem}\label{lem:smooth}
  Let $S$ be a scheme, let $X$ be an $S$-groupoid, and let $T$ be an
  \emph{affine} $X$-scheme. Let $t\in |T|$ be a point. Consider the
  following conditions. 
  \begin{enumerate}
  \item The $X$-scheme $T$ is formally smooth at
    $t$.\label{lem:item:smooth:fs_point}
  \item The $X$-scheme $T$ is formally versal at
    $t$.\label{lem:item:smooth:fv}
  \item $\Exal_X(T,\kappa(t))=0$.\label{lem:item:smooth:Exal_point}
    \end{enumerate}
  Then
  \itemref{lem:item:smooth:fs_point}$\implies$\itemref{lem:item:smooth:fv}
  and if $X$ is $\HArt$-homogeneous and $t$ is of finite type, then
  \itemref{lem:item:smooth:fv}$\implies$\itemref{lem:item:smooth:fs_point}. If
  $X$ is $\HCL$-homogeneous, $T$ is noetherian and $t$ is a \emph{closed} point,
  then
  \itemref{lem:item:smooth:fv}$\implies$\itemref{lem:item:smooth:Exal_point}. If
  $X$ is $\HrCL$-homogeneous and $t$ is a \emph{closed} point, then
  \itemref{lem:item:smooth:Exal_point}$\implies$\itemref{lem:item:smooth:fv}.
\end{lem}
Thus, assuming that an $S$-groupoid $X$ is $\HrCL$-homogeneous, we
can reformulate formal versality of an affine $X$-scheme $T$ at a
closed point $t\in |T|$ in terms of the triviality of the abelian
group $\Exal_X(T,\kappa(t))$. Understanding the set of points $U
\subset |T|$ where $\Exal_X(T,\kappa(u)) = 0$ for $u\in |U|$ will be
accomplished in the next section.
\begin{rem}
If $X$ is $\HA$-homogeneous and $\Exal_X(T,-)\equiv 0$, then $T$ is formally
smooth~\cite[Lem.\ 4.3]{hallj_openness_coh} but we will not use this. If
$\Exal_X$ commutes with Zariski localization, that is, if for any open
immersion of affine schemes $U\subseteq T$ the canonical map
$\Exal_X(T,M)\tensor_{\Gamma(\Orb_T)} \Gamma(\Orb_U)\to\Exal_X(U,M|_U)$
is bijective, then the
implications
\itemref{lem:item:smooth:fv}$\implies$\itemref{lem:item:smooth:Exal_point} and
\itemref{lem:item:smooth:Exal_point}$\implies$\itemref{lem:item:smooth:fv} also
hold for non-closed points. This is essentially what Flenner
proves in~\cite[Satz~3.2]{MR638811} as his $\mathscr{E}x(T\to X,M)$ is the
sheafification of the presheaf $U\mapsto \Exal_X(U,M|_U)$.
\end{rem}
\begin{proof}[Proof of Lemma~\ref{lem:smooth}]
  The implication
  \itemref{lem:item:smooth:fs_point}$\implies$\itemref{lem:item:smooth:fv}
  follows from the definition. The implications
  \itemref{lem:item:smooth:fv}$\implies$\itemref{lem:item:smooth:Exal_point}
  and
  \itemref{lem:item:smooth:Exal_point}$\implies$\itemref{lem:item:smooth:fv}
  are proved in~\cite[Lem.\ 4.3]{hallj_openness_coh}. The implication
  \itemref{lem:item:smooth:fv}$\implies$\itemref{lem:item:smooth:fs_point}
  follows from a similar argument: assume that $T$ is formally versal at $t$
  and let $Z_0\hookrightarrow Z$ be a square zero closed immersion of local artinian
  $X$-schemes fitting into a commutative diagram
  \[
  \xymatrix@R-0.8pc{Z_0 \ar@{(->}[d]\ar[r] & T\ar[d] \\
    Z\ar[r] & X,}
  \]
  such that the closed point $z\in |Z_0|$ is mapped to $t\in |T|$ and
  $\kappa(z)/\kappa(t)$ is a finite extension. Let $W_0$ be the image of
  $Z_0\to \spec(\Orb_{T,t})$. Then $W_0$ is a local artinian scheme with
  residue field $\kappa(t)$. As $X$ is $\HArt$-homogeneous, there is a
  commutative
  diagram
  \[
  \xymatrix@R-0.8pc{Z_0 \ar@{(->}[d]\ar[r] & W_0 \ar@{(->}[d]\ar[r] & T\ar[d] \\
    Z\ar[r] & W\ar[r] & X,}
  \]
  where $W_0\hookrightarrow W$ is a square zero closed immersion. As
  $W_0\hookrightarrow W$ is a sequence of closed immersions with kernel
  isomorphic to $\kappa(t)$, there is a lift $W\to T$ and thus a lift $Z\to T$.
\end{proof}
Combining the two lemmas above we
obtain an analogue of \cite[Prop.\ 4.2]{MR0399094}.
\begin{prop}\label{prop:artin_prop42}
  Let $S$ be a locally noetherian scheme and let $X$ be a limit
  preserving and $\HArt$-homogeneous $S$-groupoid. Let $T$ be an $X$-scheme such that
  \begin{enumerate}
    \item $T\to S$ is locally of finite type,
    \item $T\to X$ is formally versal at all points of finite type, and
    \item $T\to X$ is representable by algebraic spaces.
  \end{enumerate}
  Then $T\to X$ is formally smooth.
\end{prop}
We also obtain the following result showing that formal versality is
\'etale-local under mild hypotheses. This improves \cite[Prop.\ 4.3]{MR0399094},
which requires the existence of an obstruction theory that is compatible with
\'etale localization.
\begin{prop}\label{prop:artin_prop43}
  Let $S$ be a scheme and let $X$ be an $\HArtTriv$-homogeneous $S$-groupoid that is a
  stack in the \'etale topology. Let $T$ be an $X$-scheme and let
  $(U,u) \to (T,t)$ be a pointed
  \'etale morphism of $S$-schemes. Then formal versality at $t\in
  |T|$ implies formal versality at $u\in |U|$.
\end{prop}
\begin{proof}
Reasoning as in the proof of Lemma~\ref{lem:HArt}, we see that $X$ is
homogeneous with respect to morphisms of artinian rings with separable residue
field extensions. Arguing as in the proof of
Lemma~\ref{lem:smooth}\itemref{lem:item:smooth:fv}$\implies$\itemref{lem:item:smooth:fs_point}
we thus see that formal versality at $t\in |T|$ implies
formal versality at $u\in |U|$.
\end{proof}
Using Lemma~\ref{lem:smooth}, one can show that
Proposition~\ref{prop:artin_prop43} admits a partial converse. Indeed, if $u\in
|U|$ and $t\in |T|$ are closed, $X$ is $\HrCL$-homogeneous, $U$ and $T$ are
affine and noetherian, and $T\to X$ is representable by algebraic
spaces, then formal versality at $u\in |U|$ implies formal versality at
$t\in |T|$. This will not be used, however. 
\begin{rem}
Artin remarks~\cite[4.9]{MR0399094} that to verify the criteria for
algebraicity, it is enough to find suitable obstruction theories
\'etale-locally. We do not, however, understand the given arguments 
as \cite[Prop.\ 4.3]{MR0399094} uses the existence of a global obstruction
theory. Since
our Proposition~\ref{prop:artin_prop43} does not use obstruction theories,
it is enough to find obstruction theories \'etale-locally on $T$ in
the Main Theorem.
If one replaces semihomogeneity by
homogeneity we can thus confirm \cite[4.9]{MR0399094}.
\end{rem}
Next, we give a condition that ensures that if an $X$-scheme $T$
is formally versal at all \emph{closed} points, then it is formally versal
at all points of \emph{finite type}.
\begin{boxcnd}[Zariski localization of
  extensions] \label{cnd:Zar_loc_exal}  
  For any open immersion of {affine} $X$-schemes $p \colon U  \to
  T$, locally of finite type over $S$, and any point $u\in |U|$ of finite type,
  the natural map:
  \[
  \Exal_X(T,\kappa(u)) \to \Exal_X(U,\kappa(u))
  \]
  is surjective.
\end{boxcnd}
Note that Lemma
\ref{lem:der_exal_props_record}\itemref{lem:der_exal_props_record:item:et}
implies that Condition \ref{cnd:Zar_loc_exal} is satisfied whenever the
$S$-groupoid $X$ is $\HA$-homogeneous. It is also satisfied whenever
$S$ is Jacobson.
\begin{lem}\label{lem:Jacobson_implies_Zar_loc_exal}
Let $X$ be a Zariski $S$-stack and let $p\colon U\to T$ be an open immersion
of affine $X$-schemes. If $u\in |U|$ is a point that is closed in $T$,
then the natural map
  \[
  \Exal_X(T,\kappa(u)) \to \Exal_X(U,\kappa(u))
  \]
is an isomorphism. In particular, if $X$ is a Zariski stack and $S$ is
Jacobson, then Condition~\ref{cnd:Zar_loc_exal} is always satisfied.
\end{lem}
\begin{proof}
We construct an inverse by taking an $X$-extension $U\hookrightarrow U'$
of $U$ by $\kappa(u)$ to the gluing of $U'$ and $T\setminus \kappa(u)$ along
$U'\setminus \kappa(u)\cong U\setminus\kappa(u)$. If $S$ is Jacobson and
$T\to S$ is locally of finite type, then $T$ is Jacobson and every point of
finite type $u\in |U|$ is closed in $T$ so Condition~\ref{cnd:Zar_loc_exal} holds.
\end{proof}
We now extend the implication
\itemref{lem:item:smooth:Exal_point}$\implies$\itemref{lem:item:smooth:fv}
of Lemma~\ref{lem:smooth} to points of finite type.
\begin{prop}\label{prop:closed-pts-vs-finite-type-pts}
  Fix a scheme $S$ and an $\HrCL$-homogeneous $S$-groupoid $X$ satisfying
  Condition~\ref{cnd:Zar_loc_exal} (Zariski localization of extensions). Let
  $T$ be an affine $X$-scheme, locally of finite type over $S$, and let
  $t\in |T|$
  be a point of \emph{finite type}. If $\Exal_X(T,\kappa(t))=0$ then the
  $X$-scheme $T$ is formally versal at $t$.
\end{prop}
\begin{proof}
  Finite type points are locally closed so there exists an open affine
  neighborhood $U\subseteq T$ of $t$ such that $t\in |U|$ is closed. By
  Condition~\ref{cnd:Zar_loc_exal} we have that
  $\Exal_X(U,\kappa(t))=\Exal_X(T,\kappa(t))=0$ so the $X$-scheme $U$ is
  formally versal at $t$ by Lemma~\ref{lem:smooth}. It then follows, from the
  definition, that the $X$-scheme $T$ also is formally versal at $t$.
\end{proof}
We conclude this section by showing that $\HDVR$-homogeneity implies that
formal smoothness is stable under generizations.
Recall that a \emph{geometric discrete valuation ring} is a discrete valuation
ring $D$ such that $\spec(D)\to S$ is essentially of finite type and the
residue field is of finite type over $S$~\cite[p.~38]{MR0260746}.
\begin{defn}\label{defn:DVR-homogeneity}
Let $S$ be an excellent scheme. We say that an $S$-groupoid $X$ is
$\HDVR$-homogeneous if for any diagram of affine $S$-schemes $[\spec D
  \leftarrow \spec K \xrightarrow{{i}} \spec K']$, where $D$ is a geometric
discrete valuation ring with fraction field $K$ and ${i}$ is a nilpotent
closed immersion, the natural functor:
\[
\FIB{X}{\spec (K'\times_K D)} \to \FIB{X}{\spec K'}
\times_{\FIB{X}{\spec K}} \FIB{X}{\spec D}
\]
is an equivalence of categories. 
\end{defn}
Artin's condition [4a] of~\cite[Thm.~3.7]{MR0260746} implies
$\HDVR$-semihomogeneity and Artin's conditions [$5'$](b) and [$4'$](a,b)
of~\cite[Thm.~5.3]{MR0260746} imply $\HDVR$-homogeneity.
The following lemma is a
generalization of~\cite[Lem.~3.10]{MR0260746} from functors to categories
fibered in groupoids.
\begin{lem}\label{lem:DVR-hom_generizing}
  Let $S$ be an excellent scheme and let $X$ be a limit preserving
  $\HDVR$-homogeneous $S$-groupoid. Let $T$ be an $X$-scheme such that
  \begin{enumerate}
    \item $T\to S$ is locally of finite type,
    \item $T\to X$ is representable by algebraic spaces, and
    \item $T\to X$ is formally smooth at a point $t\in |T|$ of finite type.
  \end{enumerate}
  Then $T\to X$ is formally smooth at every generization $t'\in |T|$ of $t$.
\end{lem}
\begin{proof}
Consider a diagram of $X$-schemes
\[
\xymatrix{ Z_0 \ar@{(x->}[d] \ar[r]^g  & T \ar[d]  \\ Z
  \ar[r] \ar@{-->}[ur]   & X}
\]
where $Z_0\hookrightarrow Z$ is a closed immersion of local
artinian schemes and the image $t'=g(z_0)$ of the closed point $z_0\in |Z_0|$ is a
generization of $t\in T$ and $\kappa(z_0)/\kappa(t')$ is finite. We have to prove
that every such diagram admits a lifting as indicated by the dashed arrow.

As $X$ is limit preserving, we can factor $Z\to X$ as $Z\to W\to X$ where $W$
is an $S$-scheme of finite type. Let $h\colon T\times_X W\to T$ denote the
first projection. The pull-back $T\times_X W\to W$ is smooth at every point of
the fiber $h^{-1}(t)$ by Lemma~\ref{lem:fsmooth_pt+repr}. Let $T_t$ denote the
local scheme $\spec(\Orb_{T,t})$. It is enough to prove that $T\times_X W\to W$
is smooth at every point of $h^{-1}(T_t)$.

Let $y\in |T\times_X W|$ be a point of $h^{-1}(T_t)$. It is enough to prove
that $Y=\overline{\{y\}}$ contains a point at which $T\times_X W\to W$ is
smooth. By Chevalley's theorem, $h(Y)$ contains a constructible subset. Thus, there is a point
$w\in h(Y)\cap T_t$ such that the closure $W=\overline{\{w\}}$ in the local
scheme $T_t$ is of dimension $1$. By Lemma~\ref{lem:fsmooth_pt+repr}, it is
enough to show that $T\to X$ is formally smooth at $w$. Thus, consider a diagram
\[
\xymatrix{ \spec(K') \ar@{(x->}[d] \ar[r]^-g  & T \ar[d]  \\
\spec(K'') \ar[r] \ar@{-->}[ur]   & X}
\]
of $X$-schemes where $K''\twoheadrightarrow K'$ is a surjection of local
artinian rings such that $g(\eta)=w$ and $\kappa(\eta)/\kappa(w)$ is finite. Let
$D\subseteq K=\kappa(\eta)$ be a geometric DVR dominating $\Orb_{W,t}$ (which
exists since $\Orb_{W,t}$ is excellent). We may
then,
using $\HDVR$-homogeneity, extend the situation to a diagram
\[
\xymatrix{ \spec(K') \ar@{(x->}[d] \ar[r] & \spec(D') \ar@{(x->}[d] \ar[r]
              & T \ar[d]  \\
 \spec(K'') \ar[r] & \spec(D'') \ar[r] \ar@{-->}[ur]   & X}
\]
where $D'=D\times_K K'$ and $D''=D\times_K K''$ so that
$D'\twoheadrightarrow D$ and $D'\twoheadrightarrow D$ have nilpotent kernels.
Now, by Lemma~\ref{lem:fsmooth_pt+repr}, the pullback
$T\times_X \spec(D'')\to \spec(D'')$ is smooth at the image of
$\spec(D')$ so there is a lifting as indicated by the dashed arrow. Thus
$T\to X$ is formally smooth at $w$ and hence also at $t'$.
\end{proof}
In Lemma~\ref{lem:prorep_finhomg_affhomg}, we will show that, under mild
hypotheses, $\HDVR$-homogeneity actually implies $\HA$-homogeneity and thus
also Condition~\ref{cnd:Zar_loc_exal}.
\section{Vanishing loci for additive functors}\label{sec:vl}  
Let $T$ be a scheme. In this section we will be interested in additive
functors $F \colon \QCOH{T} \to \AB$. It is readily seen that the
collection of all such functors forms an abelian category, with
all limits and colimits computed ``pointwise''. For example, given
additive functors $F$, $G \colon \QCOH{T} \to \AB$ as well as a natural
transformation $\varphi \colon F \to G$, then $\ker \varphi \colon \QCOH{T} \to \AB$
is the functor 
\[
(\ker \varphi)(M) = \ker(F(M) \xrightarrow{\varphi(M)} G(M)).
\]
Next, we set $A = \Gamma(\Orb_T)$. Note that the natural action of $A$
on the abelian category $\QCOH{T}$ induces for every $M\in \QCOH{T}$
an action of $A$ on the abelian group $F(M)$. Thus we see that the
functor $F$ is canonically valued in the category $\MOD{A}$.
It will be
convenient to introduce the following notation: for a quasi-compact and
quasi-separated morphism of schemes $g \colon W \to T$ and 
a functor $F \colon \QCOH{T} \to \AB$, define $F_W \colon \QCOH{W} \to \AB$ to be
the functor $F_W(N) = F(g_*N)$. If $F$ is additive (resp.\ preserves
direct limits), then the same is true of $F_W$. The \emph{vanishing locus
of $F$} is the following subset \cite[\S6.2]{hallj_coh_results}:
\begin{align*}
\Van(F) &= \{ t \in |T| \suchthat F(M)=0 \quad \forall M\in \QCOH{T},
      \; \supp(M)\subset \spec(\Orb_{T,t}) \} \\
&= \{ t \in |T| \suchthat F_{\spec(\Orb_{T,t})} \equiv 0 \}
      \quad \text{(if $T$ is quasi-separated).}
\end{align*}
The main result of this section, Theorem~\ref{thm:flenner_vl}, which gives a
criterion for the set $\Van(F)$ to
be Zariski open, is essentially due to H.\ Flenner  
\cite[Lem.\ 4.1]{MR638811}. In \loccit, for an $S$-groupoid $X$ and an
affine $X$-scheme $V$, locally of finite type over $S$, a specific
result about the vanishing locus of the functor  $M \mapsto  
\Exal_X(V,M)$ is proved. In \opcit, 
the standing assumptions are that the $S$-groupoid $X$ is
\emph{semi}-homogeneous, thus the functor $M \mapsto  
\Exal_X(T,M)$ is only set-valued, which complicates matters. Since we
are assuming $\HNIL$-homogeneity of $X$, the functor $M \mapsto
\Exal_X(T,M)$ takes values in abelian groups. As we will see, this 
simplifies matters considerably. 

We now make the following trivial observation. 
\begin{lem}\label{lem:stab_gen}
  Let $T$ be a scheme and let $F \colon \QCOH{T} \to \AB$ be an additive functor.
  Then the subset $\Van(F) \subset |T|$ is stable
  under generization.  
\end{lem}
By Lemma \ref{lem:stab_gen}, we thus see that the subset
$\Van(F) \subset |T|$ will be Zariski open if we can
determine sufficient conditions on the functor $F$ and the scheme $T$ so
that the subset $\Van(F)$ is (ind)constructible. We make the
following definitions. 
\begin{defn}\label{defn:gen_cons}
  Let $T=\spec(A)$ be an affine scheme and let $F \colon \QCOH{T} \to
  \AB$ be an additive functor.
  \begin{itemize}
  \item The functor $F$ is \fndefn{bounded} if the scheme $T$ is
    noetherian and $F(M)$ is finitely generated for any finitely generated
    $A$-module $M$.
  \item The functor $F$ is \fndefn{weakly bounded} if the scheme $T$ is
    noetherian and for any integral closed subscheme
    ${i} \colon T_0 \hookrightarrow T$, the $\Gamma(\Orb_{T_0})$-module
    $F({i}_*\Orb_{T_0})$ is coherent.
  \item The functor $F$ is \fndefn{\GI} (resp.\
    \fndefn{\GS}, resp.\ \fndefn{\GB}) if there exists a dense open subset $U
    \subset |T|$ 
    such that for all points $u\in |U|$ of finite type, the map
    \[
    F(\Orb_T) \tensor_A \kappa(u) \to F(\kappa(u))
    \]
    is injective (resp.\ surjective, resp.\ bijective).
  \item The functor $F$ is
    \fndefn{\CI} (resp.\ \fndefn{\CS}, resp.\ \fndefn{\CB}) if for any
    integral closed subscheme $T_0 \hookrightarrow T$, the functor
    $F_{T_0}$ is \GI (resp.\ \GS, resp.\ \GB).
  \end{itemize}
\end{defn}
We can now state the main result of this section.
\begin{thm}[Flenner]\label{thm:flenner_vl}
  Let $T$ be an affine noetherian scheme and let $F \colon \QCOH{T} \to
  \AB$ be a half-exact, additive, and
  bounded functor that commutes with
  direct limits. If the functor $F$ is \CS,
  then the subset $\Van(F) \subset |T|$ is 
  Zariski open.
\end{thm}
Functors of the above type occur frequently in algebraic
geometry. 
\begin{ex}\label{ex:cmplx_good_properties}
  Let $T$ be an affine noetherian scheme and let $Q \in
  \DCAT^-_{\COHB}(T)$. Then, for all $i\in \Z$, the functors on
  quasi-coherent $\Orb_T$-modules given by $M\mapsto
  \Ext^i_{\Orb_T}(Q,M)$ and $M\mapsto \Tor^{\Orb_T}_i(Q,M)$ are
  additive, bounded, half-exact, commute
  with direct limits, and \CB.
\end{ex}
\begin{ex}
  Let $T$ be an affine noetherian scheme and let $p \colon X
  \to T$ be a morphism that is projective and flat. Then the functor $M\mapsto
  \Gamma(X,p^*M)$ is \CB. Indeed, one interpretation of
  the Cohomology and Base Change Theorem asserts that the functor $M\mapsto \Gamma(X,p^*M)$ is of the form
  given in Example \ref{ex:cmplx_good_properties}.  
\end{ex}
\begin{ex}\label{ex:coherent-are-bounded-and-CB}
Let $T$ be an affine noetherian scheme. An additive functor $F\colon \QCOH{T} \to
\AB$, commuting with direct limits, is \emph{coherent}~\cite{MR0212070} if
there exists a homomorphism $M\to N$ of coherent $\Orb_T$-modules such that
$F(-)=\coker(\Hom_{\Orb_T}(N,-)\to\Hom_{\Orb_T}(M,-))$. It is easily seen that
a coherent functor is \CB and bounded. Indeed, boundedness is obvious and if
$i\colon T_0\hookrightarrow T$ is an integral closed subscheme, then
$F|_{T_0}=\coker(\Hom_{\Orb_{T_0}}(i^*N,-)\to\Hom_{\Orb_{T_0}}(i^*M,-))$ and
after passing to a dense open subscheme, we may assume that $i^*N$ and $i^*M$
are flat. Then $F|_{T_0}(-)=\coker((i^*N)^\vee\to
(i^*M)^\vee)\otimes_{\Orb_{T_0}} (-)$ commutes with all tensor products. It is
well-known, and easily seen, that the functors of the previous two examples are
coherent.

Conversely, let $F\colon \QCOH{T} \to \AB$ be a half-exact bounded additive
functor that commutes with direct limits and is \CS. Then for every
integral closed subscheme $T_0\hookrightarrow T$, there is an open dense
subscheme $U_0\subset T_0$ such that $F|_{U_0}$ is coherent. In particular, for
half-exact bounded additive functors that commute with direct limits,
\CS implies \CB.
\end{ex}
The main ingredient in the proof of Theorem~\ref{thm:flenner_vl}
is a remarkable Nakayama Lemma for half-exact functors, due to
A.\ Ogus and G.\ Bergman \cite[Thm.\ 2.1]{MR0302633}. We state the
following amplification, which follows from the mild strengthening
given in \cite[Cor.\ 6.5]{hallj_coh_results} and Lemma
\ref{lem:stab_gen}.  
\begin{thm}\label{thm:nakayama}
  Let $T$ be an affine noetherian scheme and let $F \colon \QCOH{T} \to
  \AB$ be a half-exact, additive, and bounded functor that commutes with direct
  limits. Then
  \[
  \Van(F) = \{ t \in |T| \suchthat
  F(\kappa(t)) = 0 \}.
  \]
  In particular, if $F(\kappa(t))=0$ for all closed points $t\in |T|$, then $F
  \equiv 0$.
\end{thm}
\begin{rem}\label{rem:nakayama-Flenner}
Let $F$ be as in Theorem~\ref{thm:nakayama} and let $I\subset A$ be an ideal.
Then Flenner proves that the natural map
$F(M)\otimes_A \hat{A}_{/I}\to \varprojlim_n F(M/I^nM)$ is injective for every
finitely generated $A$-module $M$. In fact,
this is the special case $X=Y=\spec A$ of~\cite[Kor.~6.3]{MR638811}. The
Ogus--Bergman Nakayama lemma is an immediate consequence of the injectivity
of this map.
\end{rem}
Before we address vanishing loci of functors, the following simple
application of 
Lazard's Theorem \cite{MR0168625}, 
which appeared in \cite[Prop.\ 6.2]{hallj_coh_results}, will be a
convenient tool to have at our disposal. 
\begin{prop}\label{prop:lazardlem}
  Let $T=\spec(A)$ be an \emph{affine} scheme and let $F \colon 
  \QCOH{T} \to \AB$  be an additive functor
  that commutes with direct limits. Let $M$ and $L$ be $A$-modules. If
  $L$ is \emph{flat}, then the natural map: 
  \[
  F(M)\tensor_A L \to F(M\tensor_A L)
  \]
  is an isomorphism. In particular, for any
  $A$-algebra $B$ and any flat
  $B$-module $L$, the natural map:
  \[
  F(B) \tensor_B L \to F(L)
  \]
  is an isomorphism. 
\end{prop}
We may now prove Flenner's theorem.
\begin{proof}[Proof of Theorem~\ref{thm:flenner_vl}]
  By \cite[\textbf{IV}.1.10.1]{EGA}, the set $\Van(F)$ is open if and
  only if it is closed under generization and its intersection with
  an irreducible closed subset $T_0 \subset |T|$ contains a
  non-empty open subset or is empty. By Lemma \ref{lem:stab_gen}, we have
  witnessed the stability under generization. Thus it remains to
  address the latter claim. 

  Let $T_0 \hookrightarrow T$ be an integral closed subscheme. If
  $|T_0|\cap \Van(F) \neq \emptyset$, then the generic point $\eta
  \in |T_0|$ belongs to $\Van(F)$ (Lemma
  \ref{lem:stab_gen}), thus
  $F(\kappa(\eta)) = 0$. Since the functor $F$ is, by assumption,
  \CS,  there exists a dense open subset $U_0 \subset |T_0|$
  such that, $\forall u\in U_0$ of finite type, the map  
  $F_{T_0}(\Orb_{T_0})\tensor_{\Gamma(\Orb_{T_0})}\kappa(u) \to  
  F(\kappa(u))$ is surjective. 

  As $\kappa(\eta)$ is a quasi-coherent and flat $\Orb_{T_0}$-module,
  the natural map 
  $F_{T_0}(\Orb_{T_0})\tensor_{\Gamma(\Orb_{T_0})}
  \kappa(\eta)\to F(\kappa(\eta))$ is an isomorphism by Proposition
  \ref{prop:lazardlem}. But $\eta
  \in \Van(F)$, thus the coherent
  $\Gamma(\Orb_{T_0})$-module $F_{T_0}(\Orb_{T_0})$ is
  torsion. Hence there is a dense open subset $U_0  \subset |T_0|$
  with the property that if $u\in U_0$ is of finite type, then $F(\kappa(u)) = 
  0$. Using Theorem \ref{thm:nakayama} we infer that $U_0 \subset
  \Van(F) \cap |T_0|$.   
\end{proof}
We record for future
reference a useful lemma.
\begin{lem}\label{lem:cons_inj_surj}
  Let $T=\spec(A)$ be an affine noetherian scheme and let $F \colon 
  \QCOH{T} \to \AB$ be an additive functor.
  \begin{enumerate}
  \item\label{lem:cons_inj_surj:item:bdd_he} If the functor $F$ is
    \emph{half-exact}, then $F$ is bounded if and only if $F$ is weakly
    bounded.
  \item\label{lem:cons_inj_surj:item:bdd} If the functor $F$ is (weakly)
    bounded, then any
    additive sub-quotient functor of $F$ is (weakly) bounded. 
  \item\label{lem:cons_inj_surj:item:surj} If $F$ is \GS (resp.\ \CS),
    then so is any additive quotient functor of $F$. 
  \item\label{lem:cons_inj_surj:item:inj} If $F$ is weakly bounded and 
    \CI, then so is any additive subfunctor of $F$.
  \item Consider an exact sequence of additive functors $\QCOH{T} \to
    \AB$:
    \[
    \xymatrix{H_1 \ar[r] & H_2 \ar[r] & H_3 \ar[r] & H_4.}
    \]
    \begin{enumerate}
    \item\label{lem:cons_inj_surj:item:4surj} If $H_1$ and $H_3$ are \CS and
      $H_4$ is \CI and weakly bounded, then $H_2$ is \CS.
    \item\label{lem:cons_inj_surj:item:4inj} If $H_1$ is \CS, $H_2$ and $H_4$
      are \CI, and $H_4$ is weakly bounded, then $H_3$ is \CI.
    \end{enumerate} 
  \end{enumerate}
  If the scheme $T$ is reduced, then
  \itemref{lem:cons_inj_surj:item:inj}, \itemref{lem:cons_inj_surj:item:4surj},
  and \itemref{lem:cons_inj_surj:item:4inj} hold with \GI and \GS instead of
  \CI and \CS.
\end{lem}
\begin{proof}
  For claim \itemref{lem:cons_inj_surj:item:bdd_he}, note that any
  coherent $\Orb_T$-module $M$ admits a finite filtration whose
  successive quotients are of the form ${i}_*\Orb_{T_0}$, where 
  ${i} \colon T_0 \hookrightarrow T$ is a closed immersion with $T_0$
  integral. Induction on the length of the filtration, combined 
  with the half-exactness of the functor $F$, proves the claim.
  Claims \itemref{lem:cons_inj_surj:item:bdd} and
  \itemref{lem:cons_inj_surj:item:surj} are trivial. For
  \itemref{lem:cons_inj_surj:item:inj}, it is sufficient to prove
  the claim about \GI and we can assume that $T$ is a disjoint union of
  integral schemes. Fix an additive subfunctor $K \subset F$, then
  there is an exact 
  sequence of additive functors: $0 \to K \to F \to H \to 0$. By
  \itemref{lem:cons_inj_surj:item:bdd} we see that $H$ is
  weakly bounded and so $H(\Orb_T)$ is a finitely generated $A$-module. As $A$
  is reduced, generic
  flatness implies that there is a dense open subset $U \subset |T|$
  such that $H(\Orb_{T})_u$ is a flat $A$-module $\forall u\in U$.
  Thus, for all $u\in |U|$ the
  sequence:
  \[
  \xymatrix@-0.8pc{0 \ar[r] & K(\Orb_{T})
    \tensor_A \kappa(u) \ar[r] 
    & F(\Orb_{T}) \tensor_A \kappa(u) \ar[r] &
    H(\Orb_{T}) \tensor_A \kappa(u) \ar[r] & 0}
  \]
  is exact. By shrinking $U$, we may further assume that the map
  $F(\Orb_{T})\tensor_A \kappa(u) \to
  F(\kappa(u))$ is injective for all points $u\in |U|$ of finite type. We then
  conclude that $K$ is \GI from the
  commutative diagram:  
  \[
  \xymatrix@-0.8pc{K(\Orb_{T})
    \tensor_A \kappa(u) \ar@{(->}[r]  \ar[d] 
    & F(\Orb_{T}) \tensor_A \kappa(u)
    \ar@{(x->}[d] \\ K(\kappa(u)) \ar@{(->}[r] & F(\kappa(u)).}
  \]
  Claims
  \itemref{lem:cons_inj_surj:item:4surj} and 
  \itemref{lem:cons_inj_surj:item:4inj} follow from a similar argument
  and the $4$-Lemmas. 
\end{proof}
We conclude this section with a criterion for a functor to be
\GI (and consequently a criterion for a functor to be \CI). This will
be of use when we express Artin's criteria for algebraicity without
obstruction theories.
\begin{prop}\label{prop:cons_expanded}
  Let $T=\spec(A)$ be an affine and integral noetherian scheme with
  function field  $K$. Let $F \colon \QCOH{T} \to \AB$  be an additive functor
  that commutes with direct limits such that $F(\Orb_T)$ is a finitely
  generated $A$-module. Then $F$ is \GI if and only if the following
  condition is satisfied:
\begin{enumerate}
\renewcommand{\theenumi}{$\dagger$}
\renewcommand{\labelenumi}{\upshape{($\dagger$)}}
\item\label{prop:GI-expanded-condition}
    for any $f\in A$, any free $A_f$-module $M$, and
    $\omega \in F(M)$ such 
    that for all non-zero maps $\epsilon \colon M \to K$ we have
    $\epsilon_*\omega \neq 0$ in $F(K)$, there exists a
    dense open subset $V_\omega \subset D(f)\subset |T|$ such that for every
    non-zero map $\gamma \colon M \to \kappa(v)$, where $v\in V_\omega$ is of finite
    type, we have
    $\gamma_*\omega \neq 0$ in $F(\kappa(v))$.
\end{enumerate}
\end{prop}
\begin{proof}
  Let $M$ be a free $A_f$-module of finite rank and let
  $M^\vee=\Hom_{A_f}(M,A_f)$. Then the canonical homomorphism
  $F(A)_f\otimes_{A_f} M\to F(M)$ is an isomorphism
  (Proposition~\ref{prop:lazardlem}) so that there is a one-to-one
  correspondence between elements $\omega\in F(M)$ and homomorphisms
  $\overline{\omega}\colon M^\vee\to F(A)_f$. Moreover, $\overline{\omega}$ is
  injective if and only if $\overline{\omega}\otimes_A K\colon M^\vee\otimes_A K
  \to F(A)\otimes_A K=F(K)$ is injective and this happens exactly when
  $\epsilon_*\omega \neq 0$ in $F(K)$ for every non-zero map $\epsilon \colon M \to
  K$.

  Let $t\in |T|$ and let $\delta_t\colon F(A)\otimes_A \kappa(t)\to F(\kappa(t))$
  denote the natural map. Then condition \itemref{prop:GI-expanded-condition} can
  be reformulated as: for any free $A_f$-module $M$ of finite rank and any
  injective homomorphism $\overline{\omega}\colon M^\vee\to F(A)_f$, there exists
  a dense open subset $V_\omega\subset D(f)$ such that $\delta_t\circ
  \bigl(\overline{\omega}\otimes_A \kappa(t)\bigr)$ is injective for all points
  $t\in V_\omega$ of finite type.

  To show that \itemref{prop:GI-expanded-condition} implies that $F$ is \GI,
  choose $f\in A\setminus 0$ such that $F(A)_f$ is free,
  let $M=F(A)_f^\vee$ and let $\omega\in F(M)$ correspond to the inverse of the
  canonical isomorphism $F(A)_f\to M^\vee$. If \itemref{prop:GI-expanded-condition}
  holds, then there exists an open subset $V$ such that $\delta_t$ is injective
  for all $t\in V_\omega$, i.e., $F$ is \GI.

  Conversely, if $F$ is \GI, then there is an open subset $V$ such that
  $\delta_t$ is injective for all $t\in V$ of finite type. Given a finite free
  $A_f$-module
  $M$ and $\omega\in F(M)$, we let $V_\omega=V\cap W$ where $W\subset D(f)$ is an
  open dense subset over which the cokernel of $\overline{\omega}$ is flat.  If
  $\overline{\omega}$ is injective, it then follows that $\delta_t\circ
  \bigl(\overline{\omega}\otimes_A \kappa(t)\bigr)$ is injective for all
  $t\in V_\omega$ of finite type, that is, condition
  \itemref{prop:GI-expanded-condition} holds.
\end{proof}

\section{Openness of formal versality}\label{sec:op_fv}
As the title suggests, we now address the openness of the formally
versal locus. Let $S$ be a scheme. We isolate the following conditions
for a $\HNIL$-homogeneous $S$-groupoid $X$.
\begin{boxcnd}[Boundedness of extensions]\label{cnd:bdd_exal}
  For any affine $X$-scheme $T$, locally of finite type over
  $S$, the functor $M\mapsto \Exal_X(T,M)$ is bounded. 
\end{boxcnd}
\begin{boxcnd}[Constructibility of extensions]\label{cnd:cons_exal}
    For any affine $X$-scheme $T$, locally of finite type over
    $S$, the functor $M\mapsto \Exal_X(T,M)$ is \CS.
\end{boxcnd}
To see that these conditions are plausible, observe the following
\begin{lem}\label{lem:bdd_cons_alg_stk_true}
  Let $S$ be a locally noetherian scheme, let $X$ be an algebraic
  $S$-stack, and let $T$ be an affine $X$-scheme. Suppose that both $X$ and
  $T$ are locally of finite type over $S$. Then the functors $M \mapsto \Der_X(T,M)$ and
  $M\mapsto \Exal_X(T,M)$ are bounded and \CB. 
\end{lem}
\begin{proof}
  By \cite[Thm.\ 1.1]{MR2206635} there is a
  complex $L_{T/X} \in \DCAT_{\COHB}^-(T)$ such that for all quasi-coherent
  $\Orb_T$-modules $M$, there are natural isomorphisms $\Der_X(T,M)
  \cong \Ext^0_{\Orb_T}(L_{T/X},M)$ and $\Exal_X(T,M) \cong
  \Ext^1_{\Orb_T}(L_{T/X},M)$. The result now follows from a consideration
  of Example \ref{ex:cmplx_good_properties}. 
\end{proof}
In their current form, Conditions \ref{cnd:bdd_exal} and
\ref{cnd:cons_exal} are difficult to verify. In \S\ref{sec:rel_conds}, this
will be rectified. In any case, we can now prove
\begin{thm}\label{thm:fv_art_flenner}
  Let $S$ be a locally noetherian scheme. Let $X$ be an $S$-groupoid
  satisfying the following conditions:
\begin{enumerate}
\item $X$ is limit preserving,
\item $X$ is $\HrCL$-homogeneous,
\item Condition \ref{cnd:bdd_exal} (boundedness of extensions),
\item Condition \ref{cnd:cons_exal} (constructibility of extensions), and
\item Condition \ref{cnd:Zar_loc_exal} (Zariski localization of extensions)
\end{enumerate}
  Let $T$ be an affine $X$-scheme that is 
  locally of finite type over $S$ and let $t\in |T|$ be a \emph{closed}
  point. If $T$ is formally versal at $t\in |T|$, then $T$ is formally versal at every
  point of finite type in a Zariski open neighborhood of $t$.
  In particular, if $X$ is also $\HArt$-homogeneous and $T\to X$ is representable,
  then $T$ is formally smooth
  in a Zariski open neighborhood of~$t$.
\end{thm}
\begin{proof}
  By Condition \ref{cnd:bdd_exal} and Lemma
  \ref{lem:der_exal_props_record}, the functor $M\mapsto \Exal_X(T,M)$
  is bounded, half-exact, and preserves direct limits.
  Condition \ref{cnd:cons_exal} now implies that the functor
  $M\mapsto \Exal_X(T,M)$ satisfies the criteria of Theorem
  \ref{thm:flenner_vl}. Thus, $\Van(\Exal_X(T,-))\subset
  |T|$ is a Zariski open subset. By Lemma
  \ref{lem:smooth}\itemref{lem:item:smooth:fv}$\implies$\itemref{lem:item:smooth:Exal_point} and Theorem
  \ref{thm:nakayama}, we have that $t\in \Van(\Exal_X(T,-))$. So,
  there exists an open neighborhood $t\in U \subset |T|$ with 
  $\Exal_X(T,\kappa(u)) = 0$ for all $u\in U$. By
  Proposition~\ref{prop:closed-pts-vs-finite-type-pts}, every
  point $u\in |U|$ of finite type is formally versal. The last assertion follows
  from Lemma~\ref{lem:fsmooth_pt+repr}.
\end{proof}
\section{Automorphisms, deformations, and
  obstructions}\label{sec:aut_def_obs} 
In this section, we introduce the necessary deformation-theoretic
framework that makes it possible to verify Conditions~\ref{cnd:Zar_loc_exal},
~\ref{cnd:bdd_exal} and~\ref{cnd:cons_exal}.
To do this, we recall the formulation of
deformations and obstructions given in \cite[\S6]{hallj_openness_coh}.

Let $S$ be a scheme and let $\Phi \colon Y \to
Z$ be a $1$-morphism of $S$-groupoids. Define the category $\DEF_\Phi$ to have objects the triples
$(T,J,\eta)$, where $T$ is a $Y$-scheme, $J$ is a
quasi-coherent $\Orb_T$-module, and 
$\eta$ is a $Y$-scheme structure on the trivial $Z$-extension of $T$ by
$J$. Graphically, it is the category of 
completions of the following diagram: 
\[
\xymatrix{\ar@{(x->}[d]T \ar[r] & Y \ar[d]^{\Phi} \\ 
T\extn{J} \ar@{-->}[ur]^{\eta} \ar[r] & Z.
}
\]
There is a natural functor $\DEF_\Phi \to \SCH{Y}$ taking $(T,J,\eta)$
to $T$ and we denote the fiber of this functor over the
$Y$-scheme $T$ by $\DEF_\Phi(T)$. There is also a functor
$\DEF_\Phi(T)^\opp \to \QCOH{T}$ taking $(J,\eta)$ to $J$. We denote the fiber
of this functor over a quasi-coherent $\Orb_T$-module $J$ as
$\DEF_{\Phi}(T,J)$. Note 
that this category is naturally pointed by the trivial $Y$-extension
of $T$ by $J$. If the $1$-morphism $\Phi$ is fibered in setoids,
then the category $\DEF_\Phi(T,J)$ is discrete.
By \opcit[, Prop.\ 8.3], if $Y$ and
$Z$ are $\HNIL$-homogeneous, then the groupoid $\DEF_\Phi(T,J)$ is a 
Picard category. Denote the set of isomorphism classes of
$\DEF_\Phi(B,J)$ by $\Def_\Phi(B,J)$. Thus we obtain
$\Gamma(T,\Orb_T)$-linear functors: 
\begin{align*}
  \Def_\Phi(T,-) &\colon \QCOH{T} \to \AB,\quad J \mapsto
  \Def_\Phi(T,J)\\
  \Aut_\Phi(T,-) &\colon \QCOH{T} \to \AB,\quad J \mapsto
  \Aut_{\DEF_\Phi(T,J)}(T\extn{J}).  
\end{align*}
The Lemma that follows is an easy consequence of \cite[Lem.\
6.2]{hallj_openness_coh}.
\begin{lem}\label{lem:def_prop}
  Let $S$ be a scheme and let
  $\Phi \colon Y \to Z$ be a $1$-morphism of $\HCL$-homogeneous
  $S$-groupoids. Let ${i} \colon W \hookrightarrow T$ be a closed
  immersion of $Y$-schemes and let $N$ be a quasi-coherent
  $\Orb_W$-module. Then the natural maps:
  \[
  \Aut_\Phi(T,{i}_*N) \to \Aut_\Phi(W,N) \quad \mbox{and} \quad
  \Def_\Phi(T,{i}_*N) \to \Def_\Phi(W,N), 
  \]
  are isomorphisms. 
\end{lem}
We recall the exact sequence of \opcit[, Prop.\ 8.5], which is our
fundamental computational tool. 
\begin{prop}\label{prop:derdefseq}
  Let $S$ be a scheme and let
  $\Phi \colon Y \to Z$ be a $1$-morphism of $\HNIL$-homogeneous
  $S$-groupoids. Let $T$ be a $Y$-scheme and let $J$ be a quasi-coherent
  $\Orb_T$-module. Then 
  there is a natural $6$-term exact sequence of abelian groups:
    \[
    \xymatrix{0 \ar[r] & \Aut_{\Phi}(T,J) \ar[r] &
    \Der_{Y}(T,J) \ar[r] &
    \Der_{Z}(T,J) \ar `[r] `[l] `[dlll] `[d] [dll] &  &\\
    & \Def_{\Phi}(T,J) \ar[r] & \Exal_{Y}(T,J)
    \ar[r] & \Exal_{Z}(T,J).}
  \]
\end{prop}
We now define
${\Obs}_\Phi(T,J)=\coker\bigl(\Exal_{Y}(T,J)\to\Exal_{Z}(T,J)\bigr)$ so that we
obtain an $\Gamma(T,\Orb_T)$-linear functor
\[
\Obs_\Phi(T,-) \colon \QCOH{T} \to \AB,\quad J \mapsto \Obs_\Phi(T,J).
\]
This is the \emph{minimal obstruction theory} of $\Phi$ in the sense of
Section~\ref{sec:obs-theories}.

Recall that if $\Phi$ is $\HrCL$-homogeneous, then $\Aut_\Phi(T,-)$ and $\Def_\Phi(T,-)$
are half-exact~\cite[Cor.\ 6.4]{hallj_openness_coh}. There is no
reason to expect that $\Obs_\Phi(T,-)$ is half-exact, however.
We have the following analogue of Lemma~\ref{lem:def_prop} for obstructions.
\begin{lem}\label{lem:obs_prop}
  Let $S$ be a scheme $S$ and let  
  $\Phi \colon Y \to Z$ be a $1$-morphism of $\HCL$-homogeneous
  $S$-groupoids.  Let ${i} \colon W 
  \hookrightarrow T$ be a closed immersion of $Y$-schemes and let $N$
  be a quasi-coherent $\Orb_{W}$-module. Then there
  is a natural map $\Obs_\Phi(W,N)\to \Obs_\Phi(T,{i}_*N)$,
  which is injective and functorial in $N$.

  Moreover, if $T$ is noetherian and $\Obs_\Phi(T,{i}_*N)$ is
  finitely generated, then there exists an infinitesimal neighborhood $W_n$ of
  $W$ in $T$, i.e., a factorization of ${i}$ through a locally nilpotent
  closed immersion $j\colon W\hookrightarrow W_n$, such that
  \[
  \Obs_\Phi(W_n,j_*N)\to \Obs_\Phi(T,{i}_*N)
  \]
  is an isomorphism.
\end{lem}
\begin{proof}
Note that for $\Def$ and $\Aut$ there are always, without any homogeneity,
natural maps $\Aut_\Phi(T,{i}_*N) \to \Aut_\Phi(W,N)$ and
$\Def_\Phi(T,{i}_*N) \to \Def_\Phi(W,N)$ and $\HCL$-homogeneity equips these
with natural inverses. For $\Obs$ there is no natural map
$\Obs_\Phi(T,{i}_*N)\to \Obs_\Phi(W,N)$, but if $Y$ and $Z$ are
$\HCL$-homogeneous, we have natural maps for $\Exal_Y$ and $\Exal_Z$ in the
opposite direction and thus a natural map for $\Obs_\Phi$ as stated in the
lemma. That this map is injective follows immediately from the $\HCL$-homogeneity
of $\Phi$.

Now, given an obstruction $\omega\in \Obs_\Phi(T,{i}_*N)$, we can realize it
as a $Z$-extension $k\colon T\hookrightarrow T'$ of $T$ by ${i}_*N$. The
ideal sheaf $k_*{i}_*N\subset \Orb_{T'}$ is then annihilated by the ideal sheaf
$I$ defining the closed immersion $k\circ {i}\colon W\hookrightarrow T'$. Thus,
by the Artin--Rees lemma, we have that $(k_*{i}_*N) \cap I^n=0$ for some $n$. Let
$W_1'$ and $W_1$ be the closed subschemes of $T'$ defined by $I^n$ and
$I^n+k_*{i}_*N$. Then the morphisms in the diagram:
\[
\xymatrix@-0.2pc{ W \ar@{(->}[r]^{j_1} & W_1 \ar@{(->}[r] \ar@{(x->}[d]
  & T \ar@{(x->}[d]  \\ & W_1' \ar@{(->}[r] & T'}
\]
are closed immersions and the square is cartesian and cocartesian
in the
category of $Z$-schemes (because $Z$ is $\HCL$-homogeneous). Let
$\omega_1=[W_1\hookrightarrow W_1']\in \Obs_\Phi(W_1,(j_1)_*N)$ be the obstruction,
so that $\omega$ is the image of $\omega_1$ along the natural map given by the
first part.

We have thus shown that every element $\omega\in \Obs_\Phi(T,{i}_*N)$ is in
the image of $\Obs_\Phi(W_l,(j_l)_*N)$ for some infinitesimal neighborhood
$j_l\colon W\hookrightarrow W_l$, depending on $\omega$. Since
$\Obs_\Phi(T,{i}_*N)$ is finitely generated and $\Orb_T$ is noetherian, it
follows that there exists an infinitesimal neighborhood $j\colon
W\hookrightarrow W_n$ such that
$\Obs_\Phi(W_n,j_*N)\to \Obs_\Phi(T,{i}_*N)$
is an isomorphism.
\end{proof}
\section{Relative conditions}\label{sec:rel_conds}
Let $S$ be a locally noetherian scheme. In this section we introduce a number of
conditions
for a $1$-morphism of $\HNIL$-homogeneous $S$-groupoids $\Phi \colon Y \to
Z$. These are the relative versions of the conditions that
appear in \itemref{mainthm:item:bdd}, \itemref{mainthm:item:cons} and
\itemref{mainthm:item:Zar_loc} of the Main Theorem. 
For any of the conditions given in this
section, a $\HNIL$-homogeneous $S$-groupoid $X$ is said to have that
condition, if the structure $1$-morphism $X \to \SCH{S}$ has the
condition. These conditions are provided in the relative version so
that this paper can be
more readily seen to subsume the results of 
\cite{starr-2006}.
That these conditions are stable under composition follows from the exact
sequence of~\cite[Prop.\ 6.9]{hallj_openness_coh} and
Lemma~\ref{lem:cons_inj_surj}. Moreover, we can also bootstrap the diagonal
using~\cite[Prop.\ 6.9]{hallj_openness_coh}---the conditions for $\Aut_{X/S}$
and $\Def_{X/S}$ imply the corresponding conditions for $\Def_{\Delta_{X/S}}$
and $\Obs_{\Delta_{X/S}}$.
\begin{boxcnd}[Boundedness of automorphisms, deformations and obstructions]
For every affine $Y$-scheme $T$ that is locally of finite type over $S$ and
every integral closed subscheme ${i}\colon T_0\hookrightarrow T$,
{\renewcommand{\theenumi}{\thethm(\roman{enumi})}
\renewcommand{\labelenumi}{(\roman{enumi})}
\begin{enumerate}
\item the $\Gamma(\Orb_{T_0})$-module 
  $\Aut_\Phi(T_0,\Orb_{T_0})$ is coherent;\label{cnd:bdd_aut}
\item the $\Gamma(\Orb_{T_0})$-module 
  $\Def_\Phi(T_0,\Orb_{T_0})$ is coherent; and\label{cnd:bdd_def}
\item the $\Gamma(\Orb_{T_0})$-module 
  $\Obs_\Phi(T,{i}_*\Orb_{T_0})$ is coherent.\label{cnd:bdd_obs}
\end{enumerate}}
\end{boxcnd}
We note that Condition \ref{cnd:bdd_obs} often is satisfied for trivial reasons.
If, for example, the $S$-groupoid $Z$ satisfies Condition \ref{cnd:bdd_exal},
which is the case when $Z$ is algebraic, then $\Phi$ satisfies Condition
\ref{cnd:bdd_obs}.
\begin{lem}\label{lem:bdd}
Let $S$ be a locally
noetherian scheme  and let $\Phi\colon Y\to Z$ be a $1$-morphism of $\HrCL$-homogeneous $S$-groupoids
 satisfying Condition \ref{cnd:bdd_def}. If $Z$ satisfies
Condition \ref{cnd:bdd_exal} (boundedness of extensions), then so does $Y$.
\end{lem}
\begin{proof}
  Let $T=\spec(R)$ be an affine $Y$-scheme that is locally of finite type over $S$. By
  Lemma
  \ref{lem:der_exal_props_record}\itemref{lem:der_exal_props_record:item:he}
  the functor $M\mapsto \Exal_Y(T,M)$ is half-exact. Thus, by Lemma
 \ref{lem:cons_inj_surj}\itemref{lem:cons_inj_surj:item:bdd_he}, it is
 sufficient to prove that for any 
  integral closed subscheme ${i} \colon T_0 \hookrightarrow T$ the
  $R$-module $\Exal_Y(T,{i}_*\Orb_{T_0})$ is coherent. Now, by
  Proposition \ref{prop:derdefseq}, there is an exact sequence:  
  \[
  \xymatrix{\Def_{\Phi}(T,{i}_*\Orb_{T_0}) \ar[r] &
    \Exal_Y(T,{i}_*\Orb_{T_0}) \ar[r] &
    \Exal_Z(T,{i}_*\Orb_{T_0}).}  
  \]
  By Condition \ref{cnd:bdd_exal} the $R$-module
  $\Exal_Z(T,{i}_*\Orb_{T_0})$ is coherent. By Lemma \ref{lem:def_prop}
  we also have that $\Def_{\Phi}(T,{i}_*\Orb_{T_0}) \cong
  \Def_{\Phi}(T_0,\Orb_{T_0})$, which is a coherent
  $\Gamma(\Orb_{T_0})$-module by Condition
  \ref{cnd:bdd_def}. It now follows from the exact sequence that
  $\Exal_Y(T,{i}_*\Orb_{T_0})$ is a coherent $R$-module.
\end{proof}
Similarly, to expand Condition \ref{cnd:cons_exal} (constructibility of
extensions), we introduce the following conditions.

\begin{boxcnd}[Constructibility of automorphisms, deformations and obstructions]
\label{cnd:cons_aut+def+obs}
For every affine and irreducible
  $Y$-scheme $T$ that is locally of finite type over $S$, with reduction
  ${i}\colon T_0\hookrightarrow T$,
{\renewcommand{\theenumi}{\thethm(\roman{enumi})}
\renewcommand{\labelenumi}{(\roman{enumi})}
\begin{enumerate}
\item the functor $\Aut_\Phi(T_0,-)\colon \QCOH{T_0}\to \AB$ is
  \GB;\label{cnd:cons_aut}
\item the functor $\Def_\Phi(T_0,-)\colon \QCOH{T_0}\to \AB$ is
  \GB; and\label{cnd:cons_def}
\item the functor $\Obs_\Phi(T,{i}_*-)\colon \QCOH{T_0}\to \AB$ is
  \GI.\label{cnd:cons_obs}
\end{enumerate}}
\end{boxcnd}
\begin{lem}\label{lem:cmb_cons}
  Let $S$ be a locally noetherian scheme. Let $\Phi \colon Y \to Z$ be a $1$-morphism of $\HCL$-homogeneous
  $S$-groupoids satisfying Conditions \ref{cnd:bdd_obs},
  \ref{cnd:cons_def} and \ref{cnd:cons_obs}. If $Z$ satisfies Condition
  \ref{cnd:cons_exal} (constructibility of extensions), then so does $Y$.
\end{lem}
\begin{proof}
  Let $T$ be an affine $Y$-scheme that is locally of finite over $S$. By Proposition \ref{prop:derdefseq} there is an
  exact sequence of additive functors $\QCOH{T} \to \AB$:
  \[
  \xymatrix{\Def_{\Phi}(T,-) \ar[r] & \Exal_Y(T,-) \ar[r] & \Exal_Z(T,-) \ar[r] & \Obs_{\Phi}(T,-) \ar[r] & 0.}
  \]
  Let  ${i} \colon T_0 \hookrightarrow T$ be an integral closed subscheme. By Lemma
  \ref{lem:def_prop} we have
  that $\Def_\Phi(T_0,-) = \Def_\Phi(T,{i}_*(-))$. Condition \ref{cnd:cons_def} gives that
  $\Def_\Phi(T_0,-)$ is \GS, so the functor $\Def_\Phi(T,-)$ is \CS.
  Condition~\ref{cnd:cons_exal} says that $\Exal_Z(T,-)$ is \CS.
  The remaining two conditions together with Lemma~\ref{lem:obs_prop}
  imply that $\Obs_{\Phi}(T,-)$ is \CI and weakly bounded. In fact,
  for any integral subscheme ${i}\colon T_0\hookrightarrow T$, there is an
  infinitesimal neighborhood $j_n\colon T_0\hookrightarrow T_n$ such that
  $\Obs_{\Phi}(T_n,(j_n)_*\Orb_{T_0})\cong \Obs_{\Phi}(T,{i}_*\Orb_{T_0})$
  and $\Obs_{\Phi}(T_n,\kappa(t))\hookrightarrow \Obs_{\Phi}(T,\kappa(t))$
  is injective for all $t\in |T_0|$.  It now follows from Lemma
  \ref{lem:cons_inj_surj}\itemref{lem:cons_inj_surj:item:4surj} that the
  functor $\Exal_Y(T,-)$ is \CS.
\end{proof}
We now move on and address Condition \ref{cnd:Zar_loc_exal} (Zariski
localization of extensions).
\begin{boxcnd}[Zariski localization of automorphisms, deformations and
    obstructions]\label{cnd:Zar_loc_aut+def+obs}
For every affine and irreducible $Y$-scheme $T$ that is locally of finite type over
$S$, with reduction $T_0$, such that the generic point $\eta\in |T|$ is of
finite type, 
{\renewcommand{\theenumi}{\thethm(\roman{enumi})}
\renewcommand{\labelenumi}{(\roman{enumi})}
\begin{enumerate}
\item the natural map
  $\Aut_\Phi(T_0,\kappa(\eta)) \to \Aut_{\Phi}(U_0,\kappa(\eta))$ is
  bijective;\label{cnd:Zar_loc_aut}
\item the natural map
  $\Def_\Phi(T_0,\kappa(\eta)) \to \Def_{\Phi}(U_0,\kappa(\eta))$ is
  bijective; and\label{cnd:Zar_loc_def}
\item the natural map
  $\Obs_\Phi(T,\kappa(\eta)) \to \Obs_{\Phi}(U,\kappa(\eta))$ is
  injective\label{cnd:Zar_loc_obs}
\end{enumerate}
for every non-empty open subset $U\subset T$ with reduction $U_0$.}
\end{boxcnd}
Note that Condition~\ref{cnd:Zar_loc_aut+def+obs} trivially holds when $S$ is
Jacobson since then $U=T=\{\eta\}$.
The proof of the next
result is similar, but easier, than the proof of Lemma
\ref{lem:cmb_cons}, thus is omitted.  
\begin{lem}\label{lem:cmb_Zar_loc}
  Let $S$ be a locally noetherian scheme. Let $\Phi \colon Y\to Z$ be a
  $1$-morphism of $\HCL$-homogeneous $S$-groupoids satisfying
  Conditions \ref{cnd:bdd_obs}, \ref{cnd:Zar_loc_def} and
  \ref{cnd:Zar_loc_obs}. If $Z$ satisfies Condition
  \ref{cnd:Zar_loc_exal} (Zariski localization of extensions), then so
  does $Y$.
\end{lem}
\section{Obstruction theories}\label{sec:obs-theories}
  As in the previous section, we let $S$ be a locally noetherian scheme and let
  $\Phi \colon Y \to Z$ be a $1$-morphism of $\HNIL$-homogeneous $S$-groupoids. We
  will expand the conditions on obstructions and obtain conditions that are
  more readily verifiable.  We begin with recalling the definition of an
  $n$-step relative obstruction theory given in
  \cite[Defn.\ 6.6]{hallj_openness_coh}.

  An \fndefn{$n$-step relative obstruction
    theory for $\Phi$}, denoted 
  $\{\mathrm{o}^l(-,-), \mathrm{O}^l(-,-)\}_{l=1}^n$, is for
  each $Y$-scheme $T$, a
  sequence of additive functors (the  obstruction spaces): 
  \[
  \mathrm{O}^{l}(T,-) \colon \QCOH{T} \to \AB,\quad J \mapsto
  \mathrm{O}^{l}(T,J),\quad {l}=1,\dots,n
  \]
  as well as natural transformations of functors (the obstruction maps):
  \begin{align*}
    \mathrm{o}^1(T,-) &\colon \Exal_Z(T,-) \Rightarrow
    \mathrm{O}^1(T,-)  \\ 
    \mathrm{o}^{l}(T,-) &\colon \ker \mathrm{o}^{{l}-1}(T,-)
    \Rightarrow \mathrm{O}^{l}(T,-) \quad\mbox{for ${l}=2$, $\dots$, $n$},
  \end{align*}  
  such that the natural transformation of functors:
  \[
  \Exal_Y(T,-) \Rightarrow \Exal_Z(T,-)
  \]
  has image $\ker \mathrm{o}^n(T,-)$. Furthermore, we say that the obstruction
  theory is
\begin{itemize}
\item \fndefn{(weakly) bounded}, if for any
  affine $Y$-scheme $T$, locally of finite
  type over $S$, the obstruction spaces $M\mapsto
  \mathrm{O}^{l}(T,M)$ are (weakly) bounded functors;
\item \fndefn{Zariski- (resp.\ \'etale-) functorial} if for any open immersion
  (resp.\ \'etale morphism) of affine 
  $Y$-schemes  $g \colon V \to T$, and ${l}=1$, $\dots$, $n$, there is a natural
  transformation of functors:
  \[
  C_g^{l} \colon \mathrm{O}^{l}(T,g_*(-)) \Rightarrow
  \mathrm{O}^{l}(V,-),
  \]
  which for any quasi-coherent $\Orb_V$-modules $N$, make the
  following diagrams commute:
  \[
  \xymatrix{\Exal_X(T,g_*N) \ar[r] \ar[d] & \mathrm{O}^1(T,g_*N) \ar[d]
    & \ker \mathrm{o}^{{l}-1}(T,g_*N) \ar[r] \ar[d] &
    \mathrm{O}^{l}(T,g_*N) \ar[d] \\ \Exal_X(V,N) \ar[r] &
    \mathrm{O}^1(V,N) & \ker
    \mathrm{o}^{{l}-1}(V,N) \ar[r]  & 
    \mathrm{O}^{l}(V,N).}
  \]
  Here the leftmost map is the map $\psi$ of Lemma~\ref{lem:der_exal_props_record}
  \itemref{lem:der_exal_props_record:item:et}.
  We also require for any open immersion (resp. \'etale morphism) of affine
  schemes $h \colon W \to V$, an isomorphism of functors: 
  \[
  \alpha_{g,h}^{l} \colon C_h^{l}  \circ C_g^{l} \Rightarrow
  C_{gh}^{l}. 
  \]
\end{itemize}

\begin{rem}[Comparison with Artin's obstruction theories]
An obstruction theory in the sense of~\cite[2.6]{MR0399094} is a $1$-step
bounded obstruction theory ``that is functorial in the obvious sense''. We take
this to mean \'etale-functorial in the above sense.
Obstruction theories are usually half-exact and functorial for any
morphism, but $\Exal$ is only contravariantly functorial for \'etale morphisms
so the condition above does not make sense for arbitrary morphisms. On the
other hand, for $\HA$-homogeneous stacks, $\Exal$ is \emph{covariantly}
functorial for any morphism, cf.\ \cite[Proof of Cor.~2.5]{hallj_openness_coh}.
Also note that the minimal obstruction theory $\Obs_\Phi$ is \'etale-functorial.
\end{rem}

We have the following simple

\begin{lem}\label{lem:obs_vs_min_obs}
  Let $S$ be a locally noetherian scheme and let $\Phi \colon Y \to Z$ be a $1$-morphism of $\HNIL$-homogeneous
  $S$-groupoids. Let $\{\mathrm{o}^{l},\mathrm{O}^{l}\}_{{l}=1}^n$
  be an $n$-step relative obstruction theory for $\Phi$. Let
  $\mathrm{\widetilde{O}}^{l}(T,M)\subset \mathrm{O}^{l}(T,M)$ be the image of
  $\mathrm{o}^{{l}}(T,M)$ for ${l}=1,\dots,n$. Then
  $\{\mathrm{o}^{l},\mathrm{\widetilde{O}}^{l}\}_{{l}=1}^n$ is an $n$-step relative
  obstruction theory for $\Phi$. Moreover, let $\Obs^{l}(T,-)=\Exal_Z(T,-)/\ker
  \mathrm{o}^{l}$ and $\Obs^0(T,-)=0$. Then $\Obs^n(T,-)=\Obs_\Phi(T,-)$ and we
  have exact sequences
  \[
    \xymatrix{0 \ar[r] & \mathrm{\widetilde{O}}^{l}(T,-) \ar[r] &
    \Obs^{l}(T,-) \ar[r] & \Obs^{{l}-1}(T,-) \ar[r] & 0}
  \]
  for ${l}=1,2,\dots,n$. In particular, if the obstruction theory is (weakly)
  bounded, then so is the minimal obstruction theory $\Obs_\Phi(T,-)$.
\end{lem}

We now introduce variants of Conditions~\ref{cnd:cons_obs}
and~\ref{cnd:Zar_loc_obs} (constructibility and Zariski localization of
obstructions) in terms of an $n$-step relative obstruction theory.

\begin{boxcnd}[Constructibility of obstructions II]\label{cnd:cons_obs:n-step}
  There exists a weakly bounded $n$-step relative obstruction
  theory for $\Phi$,
  $\{\mathrm{o}^{l}(-,-),\mathrm{O}^{l}(-,-)\}_{{l}=1}^n$, such
  that for every affine irreducible $Y$-scheme $T$ that is locally of finite type over
  $S$, the obstruction spaces
  $\mathrm{O}^{l}(T,-)|_{\red{T}}\colon \QCOH{\red{T}}\to \AB$ are \GI
  for ${l}=1$, $\dots$, $n$.
\end{boxcnd}
\begin{boxcnd}[Zariski localization of obstructions
  II]\label{cnd:Zar_loc_obs:n-step}
  There exists a functorial, $n$-step relative obstruction
  theory for $\Phi$,
  $\{\mathrm{o}^{l}(-,-),\mathrm{O}^{l}(-,-)\}_{{l}=1}^n$, such
  that for every affine irreducible $Y$-scheme $T$ that is locally of finite type over
  $S$ and whose generic point $\eta\in |T|$ is of finite type,
  and every open subscheme $U\subset T$, then the
  canonical maps
  $
  \mathrm{O}^{l}(T,\kappa(\eta))
  \to \mathrm{O}^{l}(U,\kappa(\eta))
  $
  are injective for ${l}=1$, $\dots$, $n$.
\end{boxcnd}

\begin{lem}\label{lem:cons+Zarloc_obs_equiv}
  Let $S$ be a locally noetherian scheme and let $\Phi \colon Y \to Z$
  be a $1$-morphism of 
  $\HNIL$-homogeneous $S$-groupoids.
\begin{enumerate}
\item (Constructibility) $\Phi$ satisfies
  Conditions~\ref{cnd:bdd_obs} and~\ref{cnd:cons_obs} (boundedness and
  constructibility of obstructions) if and only if $\Phi$ satisfies
  \ref{cnd:cons_obs:n-step}.
  \label{lem:cons+Zarloc_obs_equiv:item:cons}
\item (Zariski localization) Conditions
  \ref{cnd:Zar_loc_obs} and \ref{cnd:Zar_loc_obs:n-step} for $\Phi$ are equivalent.
  \label{lem:cons+Zarloc_obs_equiv:item:Zar_loc}
\end{enumerate}
\end{lem}
\begin{proof}
  If $\Phi$ satisfies Conditions~\ref{cnd:bdd_obs} and~\ref{cnd:cons_obs}, then the
  minimal obstruction theory satisfies \ref{cnd:cons_obs:n-step}. Conversely,
  assume that we are given an obstruction theory $\mathrm{O}^{l}(-,-)$ as in
  \ref{cnd:cons_obs:n-step}. Let $T$ be an affine irreducible
  $Y$-scheme that is locally
  of finite type over $S$.
  Then the subfunctors $\mathrm{\widetilde{O}}^{l}(T,-)|_{\red{T}}\subset \mathrm{O}^{l}(T,-)|_{\red{T}}$
  of Lemma~\ref{lem:obs_vs_min_obs} are also \GI and weakly bounded by
  Lemma~\ref{lem:cons_inj_surj}\itemref{lem:cons_inj_surj:item:inj}. Since
  $\Obs_\Phi(T,-)$ is an iterated extension of the
  $\mathrm{\widetilde{O}}^{l}(T,-)$'s, it follows that $\Obs_\Phi(T,-)|_{\red{T}}$
  is \GI
  and weakly bounded by
  Lemma~\ref{lem:cons_inj_surj}\itemref{lem:cons_inj_surj:item:4inj}---thus
  Conditions \ref{cnd:cons_obs} and \ref{cnd:bdd_obs} hold.

  If Condition \ref{cnd:Zar_loc_obs} holds then the
  minimal obstruction theory satisfies \ref{cnd:Zar_loc_obs:n-step}. That
  Condition \ref{cnd:Zar_loc_obs:n-step} implies Condition \ref{cnd:Zar_loc_obs}
  follows from Lemma~\ref{lem:obs_vs_min_obs}.
\end{proof}
\section{Conditions on obstructions without an obstruction
theory}\label{sec:no-obs-theory}
In this section we give conditions without reference to linear obstruction
theories, just as in
\cite[{Thm.\ 5.3 [5$'$c]}]{MR0260746} and \cite{starr-2006}.

\begin{defn}[{\cite[5.1]{MR0260746}, \cite[Def.\ 2.1]{starr-2006}}]
  By a \emph{deformation situation} for $\Phi\colon Y\to Z$, we will mean data
  $(T_0\hookrightarrow T\hookrightarrow T',M)$, where $T$ is an irreducible
  affine $Y$-scheme that is locally of finite type over $S$, where $T_0=\red{T}$ is
  integral,
  where $M$ is a quasi-coherent $\Orb_{T_0}$-module, and where
  $T\hookrightarrow T'$
  is an $Z$-extension of $T$ by $M$. We say that the deformation situation is
  \emph{obstructed} if the $Z$-extension $T\hookrightarrow T'$ cannot be lifted
  to a $Y$-extension $T\hookrightarrow T'$.

  Let $\eta_0=\spec(K_0)$ denote the generic point of $T_0$, let
   $\eta=\spec(\Orb_{T,\eta_0})$, and let
   $\eta'=\spec(\Orb_{T',\eta_0})$. Thus
  $\eta\hookrightarrow \eta'$
  is a $Z$-extension of $\eta$ by $M_\eta=M\otimes_{\Orb_{T_0}} K_0$.
\end{defn}

\begin{boxcnd}[Constructibility of obstructions III]\label{cnd:cons_obs:artin}
  Given a deformation situation such that $M$ is a free $\Orb_{T_0}$-module and
  such that for every non-zero map $\epsilon \colon M_\eta \to K_0$, the resulting
  $Z$-extension $\eta\hookrightarrow \eta'_{\epsilon}$ of $\eta$ by $K_0$ is
  obstructed, then there exists a dense open subset $U_0 \subset |T_0|$ such that
  for all points $u\in U_0$ of finite type, and all non-zero maps
  $\gamma \colon M_0 \to \kappa(u)$, the
  resulting $Z$-extension $T\hookrightarrow T'_{\gamma}$ of $T$ by $\kappa(u)$
  is obstructed.
\end{boxcnd}
\begin{boxcnd}[Zariski localization of obstructions
  III]\label{cnd:Zar_loc_obs:artin} 
  For every affine and irreducible $Y$-scheme $T$ that is locally of finite type over
  $S$, such that the generic point $\eta\in |T|$ is of finite type, if a
  $Z$-extension 
  of $T$ by $\kappa(\eta)$ is obstructed, then for every
  affine open neighborhood $U\subset T$ of $\eta$, 
  the induced $Z$-extension of $U$ by $\kappa(\eta)$ is obstructed.
\end{boxcnd}

\begin{lem}\label{lem:cons+Zarloc_obs_equiv_artin}
  Let $S$ be a locally noetherian scheme and let $\Phi \colon Y \to Z$ be a $1$-morphism of $\HNIL$-homogeneous
  $S$-groupoids.
\begin{enumerate}
\item (Constructibility) If obstructions satisfy boundedness and
  Zariski localization, if $Z$ is $\HA$-homogeneous and if
  $Y$ and $Z$ are limit preserving, then
  Conditions~\ref{cnd:cons_obs} and~\ref{cnd:cons_obs:artin} for $\Phi$ are
  equivalent.\label{lem:cons+Zarloc_obs_equiv_artin:item:cons}
\item (Zariski localization)
  Conditions~\ref{cnd:Zar_loc_obs} and~\ref{cnd:Zar_loc_obs:artin} for $\Phi$
  are equivalent.\label{lem:cons+Zarloc_obs_equiv_artin:item:Zar_loc}
\end{enumerate}
\end{lem}
\begin{proof}
  Condition~\ref{cnd:Zar_loc_obs:artin} is a straightforward expansion of 
  Condition~\ref{cnd:Zar_loc_obs}.

  To see that Conditions~\ref{cnd:cons_obs} and \ref{cnd:cons_obs:artin} are
  equivalent we will use Proposition~\ref{prop:cons_expanded} with
  $F(-)=\Obs_\Phi(T,-)|_{T_0}$. Some care is needed, though, as condition
  \itemref{prop:GI-expanded-condition} of
  Proposition~\ref{prop:cons_expanded} is not quite equivalent to
  Condition~\ref{cnd:cons_obs:artin}.

  If condition~\itemref{prop:GI-expanded-condition}
  is satisfied for $F(-)=\Obs_\Phi(T,-)|_{T_0}$, i.e., if $F$
  is \GI, then it is easily seen that Condition~\ref{cnd:cons_obs:artin}
  holds for $T$. Indeed, consider a deformation situation as in
  Condition~\ref{cnd:cons_obs:artin} and let $\omega\in F(M)$ be the
  corresponding obstruction. Then $\epsilon_*\omega\in F(K_0)$ is non-zero
  since its image in $\Obs_\Phi(T_\eta,K_0)$ is non-zero. Thus, there is an
  open dense subset $U_0\subset |T_0|$ such that $\gamma_*\omega \in
  F(\kappa(u))$ is non-zero for all $u\in U_0$ of finite type and non-zero maps
  $\gamma\colon M_0\to \kappa(u)$, that is, Condition~\ref{cnd:cons_obs:artin}
  holds.

  Conversely, fix a deformation situation $(T_0\hookrightarrow T\hookrightarrow
  T',M)$ and assume that Condition~\ref{cnd:cons_obs:artin} holds for all
  deformation situations that are restrictions of the fixed one along any open
  dense subset $W\subset T$. We will prove that $F_W(-)=\Obs_\Phi(W,-)|_{W_0}$
  is \GI for sufficiently small $W$. It then follows from the Zariski
  localization Condition~\ref{cnd:Zar_loc_obs} and Lemma~\ref{lem:obs_prop}
  that $F(-)$ is \GI as
  well.

  There are natural maps $\Exal_Y(T,-)\to
  \Exal_Y(W,-)\to \Exal_Y(\eta,-)$ and similarly for $Z$.
  Since $Z$ is $\HA$-homogeneous, the maps $\Exal_Z(T,-)\to
  \Exal_Z(W,-)\to \Exal_Z(\eta,-)$ are bijective
  (Lemma~\ref{lem:der_exal_props_record}\itemref{lem:der_exal_props_record:item:et})
  so that the induced maps
  $\Obs_\Phi(T,-)\to \Obs_\Phi(W,-)\to \Obs_\Phi(\eta,-)$ are surjective.
  As $Y$ and $Z$ are limit preserving,
  it follows that for sufficiently small $W$, the homomorphism
  $\Obs_\Phi(W,K_0)\to \Obs_\Phi(\eta,K_0)$ of $K_0$-modules is bijective.

  It is now easily verified that condition \itemref{prop:GI-expanded-condition} of
  Proposition~\ref{prop:cons_expanded} holds for
  $F_W(-)=\Obs_\Phi(W,-)|_{W_0}$. Indeed, let $f$, $M$ and $\omega$ be as in
  condition \itemref{prop:GI-expanded-condition}.
  Let $V=D(f)$ and let $V\hookrightarrow V'$ be an $Z$-extension
  of $V$ by $M$ with obstruction $\omega$. Since $\Obs_\Phi(W,K_0)\to
  \Obs_\Phi(\eta,K_0)$ is injective (even bijective), it follows that
  Condition~\ref{cnd:cons_obs:artin} applies for the extension
  $V\hookrightarrow V'$. Thus there is some dense open subset $U_0\subset V_0$
  such that for every non-zero map $\gamma\colon M\to \kappa(u)$, with $u\in
  U_0$ of finite type, $\gamma_*\omega$ is non-zero in
  $\Obs_\Phi(V,\kappa(u))$. Then, by the Zariski localization
  Condition~\ref{cnd:Zar_loc_obs}, it follows that $\gamma_*\omega$ is non-zero
  in $\Obs_\Phi(W,\kappa(u))$ so that condition
  \itemref{prop:GI-expanded-condition} of
  Proposition~\ref{prop:cons_expanded} holds.
\end{proof}
\begin{rem}
If $S$ is of finite type over a Dedekind domain as in~\cite{MR0260746} (or
Jacobson), then in
Condition~\ref{cnd:cons_obs:artin} it is enough to consider closed points
$u\in U$. Indeed, in the proof of the lemma above, we are free to pass to
open dense subsets and every $S$-scheme of finite type has a dense open
subscheme which is Jacobson.
\end{rem}

\section{Proof of Main Theorem}\label{sec:mainthm}
In this section we prove the main theorem. We begin by giving a precise
definition of effectivity.

\begin{defn}\label{def:effective}
Let $X$ be a category fibered in groupoids over the category of $S$-schemes.
We say that $X$ is \emph{effective} if for every local noetherian ring
$(B,\mathfrak{m})$, such that $B$ is $\mathfrak{m}$-adically complete, with an
$S$-scheme structure $\spec B \to S$ such that the induced morphism $\spec
(B/\mathfrak{m}) \to S$ is locally of finite type, the natural functor:
\[
\FIB{X}{\spec B} \to \varprojlim_n \FIB{X}{\spec (B/\mathfrak{m}^{n})}
\]
is dense and fully faithful. Here dense means that for every object
$(\xi_n)_{n\geq 0}$ in the limit and for every $k\geq 0$, there exists an
object $\xi\in\FIB{X}{\spec B}$ such that its image in 
$\FIB{X}{\spec (B/\mathfrak{m}^{k})}$
is isomorphic to $\xi_k$.
\end{defn}

If $X$ is an algebraic stack, then the functor
$\FIB{X}{\spec B} \to \varprojlim_n \FIB{X}{\spec (B/\mathfrak{m}^{n})}$
is an equivalence of categories---thus every algebraic stack
is effective.

We now obtain the following algebraicity criterion for groupoids.  
\begin{prop}\label{prop:alg_exal}
  Let $S$ be an excellent scheme. Then an $S$-groupoid $X$ is an
  algebraic $S$-stack that is locally of finite presentation over $S$, if and
  only if it satisfies the following conditions.
  \begin{enumerate}
  \item $X$ is a stack over $(\SCH{S})_{\Et}$.
  \item $X$ is limit preserving.\label{prop:item:alg_exal:lp}
  \item $X$ is $\HArtINSEP$- and $\HrCL$-homogeneous.
  \item $X$ is effective.\label{prop:item:alg_exal:eff}
  \item The diagonal morphism $\Delta_{X/S} \colon X \to X\times_S X$ is
    representable. 
\renewcommand{\theenumi}{6\alph{enumi}} 
\setcounter{enumi}{0}
  \item Condition \ref{cnd:bdd_def} (boundedness of deformations).
  \item Condition \ref{cnd:cons_exal} (constructibility of extensions).
  \item Condition \ref{cnd:Zar_loc_exal} (Zariski localization of extensions)
  \end{enumerate}
\end{prop}
\begin{proof}
  Just as in the proof of \cite[Cor.\ 4.6]{hallj_openness_coh} conditions
  \itemref{prop:item:alg_exal:lp}--\itemref{prop:item:alg_exal:eff}---using
  only $\HArtTriv$-homogeneity---together with Condition
  \ref{cnd:bdd_def} permit the application of \cite[Thm.\
  1.5]{MR1935511}. Thus for any pair $(\spec \Bbbk \xrightarrow{x} S,
  \xi)$, where $\Bbbk$ is a field,  $x$ is a morphism locally of finite
  type, and $\xi \in \FIB{X}{x}$, there exists a pointed and affine
  $X$-scheme $(Q_\xi,q)$ such that $Q_\xi$ is locally of finite type
  over $S$. There is also an isomorphism of $X$-schemes $\spec \kappa(q)
  \cong \spec \Bbbk$ and $Q_\xi$ is a formally versal $X$-scheme at
  the closed point $q$.

  As $X$ is
  $\HrCL$-homogeneous and satisfies Condition \ref{cnd:bdd_def},
  Lemma \ref{lem:bdd} implies that $X$ satisfies Condition
  \ref{cnd:bdd_exal} (boundedness of extensions). By Lemma~\ref{lem:HArt},
  $X$ is $\HArt$-homogeneous. Using the Conditions
  \ref{cnd:Zar_loc_exal} (Zariski localization), \ref{cnd:bdd_exal}
  (boundedness of extensions), and \ref{cnd:cons_exal} (constructibility
  of extensions), together with  $\HArt$-homogeneity, it
  follows from Theorem~\ref{thm:fv_art_flenner}
  that we are free to assume---by passing
  to an affine open neighborhood of $q$---that the $X$-scheme $Q_\xi$ is
  formally smooth.

  The remainder of the proof of \cite[Cor.\
  4.6]{hallj_openness_coh} applies without change. 
\end{proof}

Before we get to the proof of the Main Theorem we must prove the
following Lemma. 
\begin{lem}\label{lem:prorep_finhomg_affhomg}
  Let $S$ be a locally noetherian scheme and let $X$ be an $S$-groupoid
  satisfying the following conditions. 
  \begin{enumerate}
  \item $X$ is an \'etale stack.
  \item $X$ is limit preserving.
  \item The diagonal morphism $\Delta_{X/S} \colon X \to X\times_S X$ is
    representable.
  \item Let $\Bbbk$ be a field, let $x \colon \spec \Bbbk \to S$ be a
    morphism that is locally of finite type, and let $\xi
    \in X(x)$. Then there exists a pointed affine $X$-scheme $(T_{\xi},t)$
    with the following properties.
    \begin{enumerate}
    \item $T_\xi$ is locally of finite type over $S$.
    \item The point $t\in |T_\xi|$ is closed and the $X$-schemes
      $\xi$ and $\spec \kappa(t)$ are isomorphic.
    \item The $X$-scheme $T_\xi$ is
      formally smooth at $t\in |T_\xi|$ (resp.\ at every generization of $t$).
    \end{enumerate}
  \end{enumerate}
  Then $X$ is $\HINT$-homogeneous (resp.\ $\HA$-homogeneous).
\end{lem}
\begin{proof}
  Since $X$ is a Zariski stack that is also limit preserving, to show that $X$
  is $\HINT$-homogeneous, it suffices, by Lemma~\ref{lem:hom_approx}, to prove
  the following assertion: for any diagram of affine $S$-schemes $[\spec A_2 \leftarrow \spec A_0
    \hookrightarrow \spec A_1]$ that are locally of finite type over $S$, where $A_1
  \twoheadrightarrow A_0$ is surjective with nilpotent kernel and $A_0\to A_2$
  is finite (resp.\ arbitrary), the canonical map
  \[
  X(\spec (A_2 \times_{A_0} A_1)) \to X(\spec A_2) \times_{X(\spec
    A_0)} X(\spec A_1) 
  \]
  is an equivalence. Let $A_3 = A_2\times_{A_0} A_1$ and
  for $j=0$, $\dots$,
  $3$ let $W_j = \spec A_j$. Then we must uniquely complete all
  commutative diagrams:
  \vspace{-6mm}
  \[
  \xymatrix@ur@-.8pc{W_0 \ar@^{(->}[r]^{i} \ar[d]_p & W_1 \ar[d]
    \ar@/^0.8pc/[ddr] & \\ W_2 
    \ar@{(->}[r] \ar@/_.8pc/[drr]& W_3 \ar@{-->}[dr] &  & \\ &
    & X }
  \vspace{-6mm}
  \]
  Since $X$ is an \'etale stack, the problem of
  constructing a map $W_3\to X$ is \'etale-local on
  $W_3$. Thus it is sufficient to construct for each point $w_3 \in
  |W_3|$ a smooth morphism of pointed schemes $(U_3,u_3) \to
  (W_3,w_3)$, together with a unique 
  map $U_3 \to X$ which is compatible with pulling back the square
  above by $U_3 \to W_3$. 

  The morphism $W_2 \to W_3$ is a nilpotent closed
  immersion, so $w_3$ is in the image of a unique point $w_2$ of
  $W_2$ and $\kappa(w_2) \cong \kappa(w_3)$. The points of $W_2$ which
  are of finite type over $S$ are dense, thus the same is true of
  $W_3$. So, we may assume 
  that the morphism $\spec \kappa(w_3) \to S$ is locally of finite
  type.

  By condition (4), there exists an affine $X$-scheme $T$ that is
  locally of finite type over $S$, which is formally smooth at a
  closed point $t\in |T|$, and the $X$-schemes $\spec \kappa(t)$
  and $\spec \kappa(w_3)$ are isomorphic. Let $W_j'=W_j\times_X
  T$ for $j=0,1,2$. By (3), the morphism 
  $T \to X$ is representable so, by
  Lemma \ref{lem:fsmooth_pt+repr}, the
  pull-back $W_j'\to W_j$ is smooth in a neighborhood of the inverse image of
  $t$ (resp.\ the inverse image of $T_t=\spec(\Orb_{T,t})$).
  Let $W'^{\mathrm{sm}}_j\subset W_j'$ be the smooth locus of
  $W_j'\to W_j$. Let $Z_2=p(W'_0\setminus {i}^{-1}(W'^{\mathrm{sm}}_1))$
  and let $W''_2=W'^{\mathrm{sm}}_2\setminus \overline{Z_2}$
  and $W''_0=p^{-1}(W''_2)$ and
  $W''_1={i}(W''_0)$ as open subsets of $W_j'^{\mathrm{sm}}$. Then all
  points above $t$ belong to the $W''_j$. Indeed, it is enough to check that
  $\overline{Z_2}$ does not contain any points above $t$. But $Z_2$ does not
  contain any points above $t$ (resp.\ $T_t$) and since $p$ is finite, $Z_2$ is
  closed (resp.\ every point of $\overline{Z_2}$ is a specialization of a point
  in $Z_2$). In particular, we
  have that $w_2$ is in the image of $W''_2$.

  By \cite[Lem.\
  A.4]{hallj_openness_coh}, there is a commutative diagram of
  $S$-schemes: 
  \[
    \xymatrix@R-1.6pc@C-.5pc{   & W''_0 \ar[dd]|!{[dl];[d]}\hole \ar[dl]
    \ar@{(->}[rr] & & 
    W''_1 \ar[dl] 
    \ar[dd] \\ W''_2
    \ar[dd]  \ar@{(->}[rr] & & W''_3 \ar[dd]  &   \\   & W_0
    \ar@{(->}[rr]|!{[r];[dr]}\hole 
    \ar[dl] & & W_1 \ar[dl]\\ 
    W_2 \ar@{(->}[rr] & &W_3  &   } 
  \]
  where all faces of the cube are cartesian, the top and bottom faces
  are cocartesian, and the map $W_3'' \to W_3$ is flat. By \cite[Lem.\
  A.5]{hallj_openness_coh}, the morphism is $W_3'' \to W_3$ is
  smooth. Since the top square is cocartesian, and there are
  compatible maps $W_j'' \to T$ for $j\neq 3$, there is a
  uniquely induced map $W_3'' \to T$. Taking the composition of
  this map with $T \to X$, we obtain a map $W_3'' \to X$
  which is compatible with the data. This map is unique because the
  diagonal of $X$ is representable. As $W_3''\to W_3$ is a smooth
  neighborhood of $w_3$, the result follows.
\end{proof}
We are now ready to prove the Main Theorem. 
\begin{proof}[Proof of Main Theorem] 
  Repeating the bootstrapping techniques of the proof of \cite[Thm.\
  A]{hallj_openness_coh}, it is sufficient to prove the result in the 
  case where the diagonal $1$-morphism $\Delta_{X/S} \colon X \to X\times_S
  X$ is representable.

  As in the first part of the proof of Proposition \ref{prop:alg_exal}, we see
  that for every point $x\colon \spec(\Bbbk)\to X$ that is of finite type
  over $S$, there exists 
  an affine $X$-scheme $(Q_\xi,q)$ such that $Q_\xi$ is locally of finite type
  over $S$, together with an isomorphism of $X$-schemes $\spec \kappa(q) \cong \spec
  \Bbbk$, and $Q_\xi$ is a formally versal $X$-scheme at the closed point
  $q$. Now since $X$ is $\HArt$-homogeneous (Lemma~\ref{lem:HArt}), it follows
  that $Q_\xi$ is
  formally smooth at the closed point $q$ by Lemma~\ref{lem:smooth}. If $X$ is
  $\HDVR$-homogeneous, then $Q_\xi$ is even formally smooth at every
  generization of $q$ by Lemma~\ref{lem:DVR-hom_generizing}. Then,
  by Lemma \ref{lem:prorep_finhomg_affhomg}, we see that the $S$-groupoid $X$ is
  $\HINT$-homogeneous (and $\HA$-homogeneous if $X$ is $\HDVR$-homogeneous)
  and thus also $\HrCL$-homogeneous.

  So, by Proposition \ref{prop:alg_exal}, it remains
  to show that the hypotheses of the Theorem guarantee that Conditions
  \ref{cnd:cons_exal} and \ref{cnd:Zar_loc_exal} (constructibility and Zariski
  localization of extensions) hold for $X$. We saw that if $X$ is
  $\HDVR$-homogeneous, then $X$ is $\HA$-homogeneous and
  Condition \ref{cnd:Zar_loc_exal} holds by
  Lemma~\ref{lem:der_exal_props_record}\itemref{lem:der_exal_props_record:item:et}.
  Likewise, if $S$ is Jacobson, then Condition \ref{cnd:Zar_loc_exal} holds by
  Lemma~\ref{lem:Jacobson_implies_Zar_loc_exal}.

  Now, by Lemmata~\ref{lem:bdd_cons_alg_stk_true}
  and~\ref{lem:der_exal_props_record}\itemref{lem:der_exal_props_record:item:et}
  we have that Conditions~\ref{cnd:bdd_exal}, \ref{cnd:cons_exal} and
  \ref{cnd:Zar_loc_exal} (boundedness, constructibility and Zariski localization
  of extensions) hold for $S$. Trivially, Condition~\ref{cnd:bdd_obs}
  (boundedness of obstructions) then holds for $X$. By
  Lemmata~\ref{lem:cons+Zarloc_obs_equiv}\itemref{lem:cons+Zarloc_obs_equiv:item:Zar_loc},
  \ref{lem:cons+Zarloc_obs_equiv_artin}\itemref{lem:cons+Zarloc_obs_equiv_artin:item:Zar_loc} and \ref{lem:cmb_Zar_loc} we see that
  the hypothesis \itemref{mainthm:item:Zar_loc} implies
  Condition~\ref{cnd:Zar_loc_exal}.
  Similarly, by Lemmata~\ref{lem:cons+Zarloc_obs_equiv}\itemref{lem:cons+Zarloc_obs_equiv:item:cons},
  \ref{lem:cons+Zarloc_obs_equiv_artin}\itemref{lem:cons+Zarloc_obs_equiv_artin:item:cons} and \ref{lem:cmb_cons},
  the hypothesis \itemref{mainthm:item:cons} implies
  Condition~\ref{cnd:cons_exal}.
  We may thus apply Proposition~\ref{prop:alg_exal} to conclude that $X$ is an
  algebraic stack that is locally of finite presentation over $S$.
\end{proof}
\section{Comparison with other criteria}\label{sec:comparison}
In this section we compare our algebraicity criterion with Artin's
criteria~\cite{MR0260746,MR0399094}, Starr's criterion~\cite{starr-2006}, the
criterion of the first author~\cite{hallj_openness_coh}, the criterion in
the stacks project~\cite{stacks-project}, and Flenner's criterion
for openness of versality~\cite{MR638811}.

\subsection{Artin's algebraicity criterion for functors}
In~\cite[Thm.\ 5.3]{MR0260746}
Artin assumes [$0'$]=\itemref{mainthm:item:fppf-stack} (fppf stack),
[$1'$]=\itemref{mainthm:item:lp} (limit preserving) and
[$2'$]=\itemref{mainthm:item:eff} (effectivity). Further [$4'$](b)+[$5'$](a) is
$\HNIL$-homogeneity for irreducible schemes, which implies
\itemref{mainthm:item:homog}. His [$4'$](a)+(c) is boundedness,
Zariski-localization and constructibility of deformations (Conditions~\ref{cnd:bdd_def}, \ref{cnd:Zar_loc_def} and~\ref{cnd:cons_def}).
His [$5'$](c) is Condition~\ref{cnd:cons_obs:artin} (constructibility of
obstructions).
Finally, [$5'$](b) together with [$4'$](a) and [$4'$](b) implies
$\HDVR$-homogeneity and hence \itemref{mainthm:item:Zar_loc}. Conditions on
automorphisms are of course redundant for functors. Condition [$3'$](a) is only
used to assure that the resulting algebraic space is locally separated
(resp.\ separated) and condition [$3'$](b) guarantees that it is
quasi-separated. If one  is willing to accept non quasi-separated algebraic
spaces, no separation assumptions are necessary.

\subsection{Artin's algebraicity criterion for stacks}
Let us begin with correcting two typos in the statement of
\cite[Thm.\ 5.3]{MR0399094}. In (1) the condition should be that (S$1'$,2)
holds for $F$, not merely (S1,2), and in (2) the canonical map should be fully
faithful with dense image, not merely faithful with dense image. Otherwise it
is not possible to bootstrap and deduce algebraicity of the diagonal.

Artin assumes that $X$ is a stack for the \'etale topology \opcit[, (1.1)], and
that $X$ is limit preserving. He assumes (1) that the Schlessinger conditions
(S$1'$,2) hold and boundedness of automorphisms. In our terminology, (S$1'$) is
$\HrCL$-homogeneity, which implies $\HArtTriv$-homogeneity,
our~\itemref{mainthm:item:homog}. The other two conditions are exactly
boundedness of automorphisms and
deformations~\itemref{mainthm:item:bdd}. Artin's condition (2) is our
\itemref{mainthm:item:eff} (effectivity). Artin's condition (3) is \'etale
localization and constructibility of automorphisms, deformations and
obstructions, and compatibility with completions for automorphisms and
deformations. The constructibility condition is slightly stronger than
our~\itemref{mainthm:item:cons} and the \'etale localization condition implies
the much weaker~\itemref{mainthm:item:Zar_loc}. We do not use compatibility
with completions. Finally, Artin's condition (4) implies that the double
diagonal of the stack is quasi-compact and this condition can be omitted if we
work with stacks without separation conditions.
Thus~\cite[Thm.\ 5.3]{MR0399094} follows from our main theorem, except that
Artin only assumes that the groupoid is a stack in the \'etale topology. This
is related to the issue when comparing formal versality to formal smoothness
mentioned in the introduction and discussed in the beginning of
Section~\ref{sec:fv_fs}.

\begin{rem}
That automorphisms and deformations are sufficiently compatible with
completions for Artin's proof to go through actually follows from the other
conditions. In fact, let $A$ be a noetherian local ring with maximal ideal
$\mathfrak{m}$, let $T=\spec(A)$ and let $T\to X$ be given. Then the
injectivity of the comparison map
\[
\varphi\colon \Def_{X/S}(T,M)\otimes_A \hat{A}\to
\varprojlim_n \Def_{X/S}(T,M/\mathfrak{m}^n)
\]
for a finitely generated $A$-module $M$ follows from the boundedness of
$\Def_{X/S}(T,-)$, see Remark~\ref{rem:nakayama-Flenner}. If $T\to X$ is
formally versal, then $\varphi$ is also surjective. Indeed, from (S1) it
follows that $\Der_S(T,M/\mathfrak{m}^n)\to \Def_{X/S}(T,M/\mathfrak{m}^n)$ is
surjective for all $n$, so the composition $\Der_S(T,M)\otimes_A \hat{A}\cong
\varprojlim_n \Der_S(T,M/\mathfrak{m}^n)\to \varprojlim_n
\Def_{X/S}(T,M/\mathfrak{m}^n)$, which factors through $\varphi$, is surjective.
\end{rem}

The variant~\cite[Prop.~1.1]{starr-2006}, due to Starr, has the same conditions
as~\cite[Thm.\ 5.3]{MR0399094} except that it is phrased in a relative
setting. From Section~\ref{sec:rel_conds}, it is clear that our conditions can
be composed. The salient point is that with $\HrCL$-homogeneity (or even with
just (S1), i.e., $\HrCL$-semihomogeneity, as in~\cite{MR638811}), there is
always a linear minimal obstruction theory. There is further an exact sequence
relating the minimal obstruction theories for the composition of two
morphisms~\cite[Prop.~6.9]{hallj_openness_coh}.
Thus~\cite[Prop.~1.1]{starr-2006} also follows from our main theorem.

\subsection{The criterion~\cite{hallj_openness_coh} using coherence}
There are two differences between \cite[Thm.~A]{hallj_openness_coh} and
our main theorem. The first is that Condition~\itemref{mainthm:item:homog}
is strengthened to $\HA$-homogeneity. As this includes $\HDVR$-homogeneity,
\itemref{mainthm:item:Zar_loc} becomes redundant. Zariski localization also
follows immediately from $\HA$-homogeneity without involving
$\HDVR$-homogeneity, see discussion after Condition~\ref{cnd:Zar_loc_exal}. We
thus have the following version of our main theorem.
\begin{thm}
Let $S$ be an excellent scheme. Then a category $X$ that is fibered in groupoids over
the category of $S$-schemes, $\SCH{S}$, is an algebraic stack that is locally of
finite presentation over $S$, if and only if it satisfies the following
conditions.
\begin{enumerate}
  \item[($1'$)] $X$ is a stack over $(\SCH{S})_{\Et}$.
  \item[(2)] $X$ is limit preserving.
  \item[($3''$)] $X$ is $\HA$-homogeneous.
  \item[(4)] $X$ is effective.
\renewcommand{\theenumi}{5\alph{enumi}} 
\setcounter{enumi}{0}
  \item Automorphisms and deformations are bounded
    (Conditions~\ref{cnd:bdd_aut} and~\ref{cnd:bdd_def}).
  \item Automorphisms, deformations and obstructions are constructible
    (Conditions~\ref{cnd:cons_aut} and~\ref{cnd:cons_def} and either
    Condition~\ref{cnd:cons_obs}, \ref{cnd:cons_obs:n-step}
    or~\ref{cnd:cons_obs:artin}).
\end{enumerate}
\end{thm}
The second difference is that~\itemref{mainthm:item:bdd}
and~\itemref{mainthm:item:cons} are replaced with the condition that
$\Aut_{X/S}(T,-)$, $\Def_{X/S}(T,-)$, $\Obs_{X/S}(T,-)$ are \emph{coherent}
functors. This implies that the functors are bounded and $\CB$
(Example~\ref{ex:coherent-are-bounded-and-CB}), hence satisfy
\itemref{mainthm:item:bdd} and~\itemref{mainthm:item:cons}.

\subsection{The criterion in the Stacks project}
In the Stacks project, the basic version of Artin's
axiom~\cite[\spref{07XJ},\spref{07Y5}]{stacks-project} requires that
\begin{enumerate}
\item[{[0]}] $X$ is a stack in the \'etale topology,
\item[{[1]}] $X$ is limit preserving,
\item[{[2]}] $X$ is $\HArt$-homogeneous (this is the Rim--Schlessinger condition RS),
\item[{[3]}] $\Aut_{X/S}(\spec(k),k)$ and $\Def_{X/S}(\spec(k),k)$ are finite dimensional,
\item[{[4]}] $X$ is effective, and
\item[{[5]}] $X$, $\Delta_X$ and $\Delta_{\Delta_X}$ satisfy openness of versality.
\end{enumerate}
There is also a criterion for when $X$ satisfies openness of
versality~\cite[\spref{07YU}]{stacks-project} using naive obstruction
theories with finitely generated cohomology groups. This uses the
(RS*)-condition which is our
$\HA$-homogeneity~\cite[\spref{07Y8}]{stacks-project}.
The existence of the
naive obstruction theory implies that $\Aut_{X/S}(T,-)$,
$\Def_{X/S}(T,-)$, $\Obs_{X/S}(T,-)$ are bounded and $\CB$
(Example~\ref{ex:cmplx_good_properties}), hence satisfy
\itemref{mainthm:item:bdd} and~\itemref{mainthm:item:cons}.

In~\cite{stacks-project}, the condition that the base scheme $S$ is excellent
is replaced with the condition that its local rings are $G$-rings. In our
treatment, excellency enters at two places: in the application of
N\'eron--Popescu desingularization in Proposition~\ref{prop:alg_exal}
via~\cite{MR1935511} and in the context of $\HDVR$-homogeneity
in Lemma~\ref{lem:DVR-hom_generizing}. In
both cases, excellency can be replaced with the condition that the local
rings are $G$-rings without modifying the proofs.

\subsection{Flenner's criterion for openness of versality}
Flenner does not give a precise analogue of our main theorem, but his main
result~\cite[Satz~4.3]{MR638811} is a criterion for the openness of versality.
In his criterion he has a limit preserving $S$-groupoid which satisfies
(S1)--(S4). The first condition (S1) is identical to Artin's condition (S1),
i.e., $\HrCL$-semihomogeneity. The second condition (S2) is boundedness and
Zariski localization of deformations. The third condition (S3) is boundedness
and Zariski localization of the minimal obstruction theory. Finally (S4) is
constructibility of deformations and obstructions. The Zariski localization
condition is incorporated in the formulation of (S3) and (S4) which deals with
\emph{sheaves} of deformation and obstructions modules. His (S2)--(S4) are
marginally stronger than our conditions, for example, treating arbitrary
schemes instead of irreducible schemes. Theorem~\cite[Satz~4.3]{MR638811} thus
becomes the first part of Theorem~\ref{thm:fv_art_flenner}, in the view of
Section~\ref{sec:rel_conds}, except that we assume $\HrCL$-homogeneity instead
of $\HrCL$-semihomogeneity. This is a pragmatic choice that simplifies matters
since $\Exal_X(T,M)$ becomes a module instead of a pointed set. Also, in any
algebraicity criterion, we would need homogeneity to deduce that the
diagonal is algebraic and, conversely, if the diagonal is algebraic, then
semihomogeneity implies homogeneity.

\subsection{Criterion for local constructibility}
There is a useful criterion for when a sheaf (or a stack) is locally
constructible, that is, when it corresponds to an \'etale algebraic space
(or algebraic stack)~\cite[VII.7.2]{MR0407011}:
%
\begin{thm}\label{thm:etale-crit}
Let $S$ be an excellent scheme. Then a category $X$ that is
fibered in groupoids over $\SCH{S}$, is an
algebraic stack that is \'etale over $S$, if and only if it
satisfies the following conditions.
\begin{enumerate}
  \item \label{etalethm:item:fppf-stack}
    $X$ is a stack over $(\SCH{S})_{\Et}$.
  \item \label{etalethm:item:lp}
    $X$ is limit preserving.
  \item \label{etalethm:item:eff}
    $X(B)\to X(B/\mathfrak{m})$ is an equivalence of categories for every local
    noetherian ring $(B,\mathfrak{m})$, such that $B$ is
    $\mathfrak{m}$-adically complete, with an $S$-scheme structure $\spec B \to
    S$ such that the induced morphism $\spec (B/\mathfrak{m}) \to S$ is of
    finite type.
\end{enumerate}
\end{thm}
The necessity of the conditions is clear. That the conditions are
sufficient can be proven directly as follows. Let $j\colon
(\SCH{S})_{\Et}\to S_{\et}$ denote the morphism of topoi corresponding to the
inclusion of the small \'etale site into the big \'etale site. It is enough
to prove that $j^{-1}j_*X\to X$ is an equivalence. As $X$ is limit preserving,
it is enough to verify that $f^*(X|_{S_\et})\to X|_{T_{\et}}$ is an equivalence
for every morphism $f\colon T\to S$ locally of finite type, and this can be
checked on stalks at points of finite type. Therefore, it suffices to prove
that $X(B)\to X(B/\mathfrak{m})$ is an equivalence when $B$ is the henselization
of $\Orb_{T,t}$, for every $t\in |T|$ of finite type. This follows from general
N\'eron--Popescu desingularization and the
three conditions.

A proof more in the lines of this paper goes as follows:
from~\itemref{etalethm:item:eff} it follows that: $X$ is $\HArt$-homogeneous;
$X$ is effective; and $X\to S$ is formally \'etale at every point of finite
type. In particular, $\Aut_{X/S}(T,N)=\Def_{X/S}(T,N)=\Obs_{X/S}(T,N)=0$ for
every $X$-scheme $T$ that is of finite type over $S$ and every quasi-coherent
$\Orb_T$-module $N$ with support that is artinian (use
Lemmata~\ref{lem:def_prop} and~\ref{lem:obs_prop}). Thus,
$\Aut_{X/S}(T,-)=\Def_{X/S}(T,-)=0$ by
Theorem~\ref{thm:nakayama}. Theorem~\ref{thm:etale-crit} would follow from the
main theorem if we also can show that $\Obs_{X/S}(T,-)=0$. As we do not yet
know that $\Obs_{X/S}(T,-)$ is half-exact, it is apparently difficult to deduce
that $\Obs_{X/S}(T,-)=0$ without invoking Popescu desingularization. A more
elementary approach, that does not rely on the main theorem, is to note
that given an $X$-scheme $T$ that is locally of finite presentation over $S$, and a
point $t\in |T|$ of finite type, then $T\to X$ is formally smooth at $t$ if and
only $T\to S$ is formally smooth at $t$. Thus, openness of formal smoothness
for $T\to X$ follows.
\appendix
\section{Approximation of integral morphisms}
In this appendix, we give an approximation result for integral homomorphisms. It
is somewhat technical since the properties that we need---surjective and
surjective with nilpotent kernel---cannot be deduced for an arbitrary
approximation. In fact, the approximation has to be built with these properties in
mind.
\begin{lem}\label{lem:integral-approx}
Let $A$ be a ring, let $B$ be an $A$-algebra and let $C$ be an $B$-algebra.
Assume that $B$ and $C$ are integral $A$-algebras. Then there exists a filtered
system $(B_\lambda\to C_\lambda)_\lambda$ of finite and finitely presented
$A$-algebras, with direct limit $B\to C$. In addition, if $A\to B$ (resp.\ $B\to C$,
resp.\ $A\to C$) has one of the properties:
\begin{enumerate}
\item surjective,
\item surjective with nilpotent kernel,
\end{enumerate}
then $A\to B_\lambda$ (resp.\ $B_\lambda\to C_\lambda$, resp.\ $A\to
C_\lambda$) has the corresponding property.
\end{lem}
\begin{proof}
We begin by writing $B=\varinjlim_{\lambda\in\Lambda} B^\circ_\lambda$ and
$C=\varinjlim_{\lambda\in\Lambda} C^\circ_\lambda$ as direct limits of finitely
generated subalgebras.  We may then replace $C^\circ_\lambda$ with the
$C$-subalgebra generated by the images of $B^\circ_\lambda$ and
$C^\circ_\lambda$ so that we have homomorphisms $B^\circ_\lambda\to
C^\circ_\lambda$ for all $\lambda$. If $B\to C$ is surjective, then we let
$C^\circ_\lambda$ be the image of $B^\circ_\lambda\to C$. It is now easily
verified that if $A\to B$ (resp.\ $B\to C$, resp.\ $A\to C$) is surjective or
surjective with nilpotent kernel then so is $A\to B^\circ_\lambda$
(resp.\ $B^\circ_\lambda\to C^\circ_\lambda$, resp.\ $A\to C^\circ_\lambda$).

For every $\lambda$, choose surjections $P_\lambda\to B^\circ_\lambda$ and
$Q_\lambda\to C^\circ_\lambda$ where $P_\lambda$ and $Q_\lambda$ are finite and
finitely presented $A$-algebras. We may assume that we have homomorphisms
$P_\lambda\to Q_\lambda$ compatible with $B^\circ_\lambda\to C^\circ_\lambda$
and if $B\to C$ is surjective, then we take $P_\lambda=Q_\lambda$. For any
finite subset $L\subseteq \Lambda$ let $P_L=\bigotimes_{\lambda\in L}
P_\lambda$ and $Q_L=\bigotimes_{\lambda\in L} Q_\lambda$, where the tensor
products are over $A$.

For fixed $L\subseteq \Lambda$ choose finitely generated ideals $I_L\subseteq
\ker(P_L\to B)$ and $I_L Q_L \subseteq J_L\subseteq \ker(Q_L\to C)$ and let
$B_L=P_L/I_L$ and $C_L=Q_L/J_L$. If $A\to B$ (resp.\ $A\to C$) is surjective,
then for sufficiently large $I_L$ (resp.\ $J_L$), we have that $A\to B_L$
(resp.\ $A\to C_L$) is surjective. If $B\to C$ is surjective, then by
construction $P_L=Q_L$ so that $B_L\to C_L$ is surjective. If $B\to C$ has
nilpotent kernel, with nilpotency index $n$, then we replace $I_L$ with
$I_L+J_L^n$ so that $B_L\to C_L$ has nilpotent kernel.

Consider the set $\Xi$ of pairs $\xi=(L,I_L,J_L)$ where $L\subseteq \Lambda$ is
a finite subset, and $I_L\subseteq P_L$ and $J_L\subseteq Q_L$ are finitely
generated ideals as in the previous paragraph. Then $(B_L\to C_L)_\xi$ is a
filtered system of finite and finitely presented $A$-algebras with direct
limit
$(B\to C)$ which satisfies the conditions of the lemma.
\end{proof}
Fix a scheme $S$ and consider the category of diagrams $[Y \xleftarrow{f} X
  \xrightarrow{{i}} X']$ of $S$-schemes. We say that a morphism
$[Y_1 \xleftarrow{f_1} X_1 \xrightarrow{{i}_1} X'_1]\to 
[Y_2 \xleftarrow{f_2} X_2 \xrightarrow{{i}_2} X'_2]$ is affine if the
components $Y_1\to Y_2$, $X_1\to X_2$ and $X'_1\to X'_2$ are affine. Given an
inverse system of diagrams with affine bonding maps the inverse
limit then exists and is calculated component by component.
\begin{prop}\label{prop:pushout-approx}
Let $S$ be an affine scheme and let $P\in \{\HNIL,\HCL,\HrNIL,\HrCL,\HINT,\HA\}$
(cf.\ Section~\ref{sec:homogeneity}). Let $\mathbf{W}=[Y \xleftarrow{f} X
  \xrightarrow{{i}} X']$ be a diagram of affine $S$-schemes where
${i}$ is a nilpotent closed immersion, and $f$ is $P$.
Then $\mathbf{W}$ is an inverse limit of diagrams
$\mathbf{W}_\lambda=[Y_\lambda \xleftarrow{f_\lambda} X_\lambda
  \xrightarrow{{i}_\lambda} X'_\lambda]$ of affine finitely presented
$S$-schemes where ${i}_\lambda$ is a nilpotent closed immersion, and
$f_\lambda$ is $P$.
Moreover, if we let $Y'=Y\amalg_X X'$ and
$Y'_\lambda=Y_\lambda\amalg_{X_\lambda} X'_\lambda$ denote the push-outs, then
$Y'=\varprojlim_{\lambda\in\Lambda} Y'_\lambda$.
\end{prop}
\begin{proof}
\newcommand{\Ybar}{\mathrlap{\bar{\phantom{Y}}}Y}
We will begin by looking at the induced diagram $[Y \xrightarrow{j} Y'
  \xleftarrow{g} X']$. As $j$ is a nilpotent closed immersion it follows that
$g$ has property $P$. The first step will be to write this diagram as an
inverse limit of diagrams $[Y_\lambda \xrightarrow{j_\lambda} \Ybar'_\lambda
  \xleftarrow{g_\lambda} X'_\lambda]$ of finite presentation over $S$ where
$j_\lambda$ is a nilpotent closed immersion and $g_\lambda$ has property $P$.
To this end, we begin by writing $Y'$ as an inverse limit of finitely
presented $S$-schemes $\Ybar'_{\alpha}$. By
Lemma~\ref{lem:integral-approx} we may also write $g\colon X'\to Y'$
(resp.\ $j\colon Y\to Y'$) as an inverse limit of finitely presented
$P$-morphisms $X'_\beta\to Y'$ (resp.\ finitely presented nilpotent closed
immersions $Y_\gamma\to Y'$).
For every pair $(\beta,\gamma)$ there is (\cite[\textbf{IV}.8.10.5]{EGA}) an
$\alpha=\alpha(\beta,\gamma)$, and a cartesian diagram
\[
\xymatrix@-0.7pc{Y_\gamma \ar@{(->}[r]\ar[d]
  & Y'\ar[d] & X'_\beta\ar[l]\ar[d] \\
  Y_{\alpha\beta\gamma}\ar@{(->}[r]
  & \Ybar'_\alpha & X'_{\alpha\beta\gamma}.\ar[l]}
\]
where $X'_{\alpha\beta\gamma}\to \Ybar'_{\alpha}$ is a finitely
presented $P$-morphism and $Y_{\alpha\beta\gamma}\to \Ybar'_{\alpha}$
is a nilpotent closed immersion.

For every $\alpha\geq \alpha(\beta,\gamma)$ we also let $[Y_{\alpha\beta\gamma}
  \rightarrow \Ybar'_\alpha \leftarrow X'_{\alpha\beta\gamma}]$ denote
the pull-back along $\Ybar'_\alpha\to
\Ybar'_{\alpha(\beta,\gamma)}$. Let $I=\{(\beta,\gamma,\alpha)\}$ be the
set of indices such that $\alpha>\alpha(\beta,\gamma)$. For every finite subset
$J\subset I$, we let
%
$$\Ybar'_J=\prod_{\mathclap{(\beta,\gamma,\alpha)\in J}}
   \Ybar'_{\alpha},\quad
Y_J=\prod_{\mathclap{(\beta,\gamma,\alpha)\in J}}
   Y_{\alpha\beta\gamma},\quad\text{and}\quad
X'_J=\prod_{\mathclap{(\beta,\gamma,\alpha)\in J}} X'_{\alpha\beta\gamma}$$
where the products are taken over $S$.
The finite subsets $J\subset I$ form a partially ordered set
under inclusion and the induced morphisms:
$$\Ybar'\to \varprojlim_J \Ybar'_J,\quad
Y\to \varprojlim_J Y_J,\quad\text{and}\quad
X'\to \varprojlim_J X'_J$$
are closed immersions. Now, let $K_{Y_J}=\ker(\Orb_{Y_J}\to (g_J)_*\Orb_{Y})$
and similarly for $K_{\Ybar'_J}$ and $K_{X'_J}$. Note that
$K_{\Ybar'_J}\Orb_{Y_J}\subset K_{Y_J}$ and
$K_{\Ybar'_J}\Orb_{X'_J}\subset K_{X'_J}$. We then let
$\Lambda=\{(J,R_{Y_J},R_{\Ybar'_J},R_{X'_J})\}$ where $J\subset I$ is a
finite subset and $R_{Y_J}\subset K_{Y_J}$, $R_{\Ybar'_J}\subset
K_{\Ybar'_J}$ and $R_{X'_J}\subset K_{X'_J}$ are finitely generated
ideals such that $R_{\Ybar'_J}\Orb_{Y_J}\subset R_{Y_J}$ and
$R_{\Ybar'_J}\Orb_{X'_J}\subset R_{X'_J}$. For every $\lambda\in
\Lambda$ we put
$$\Ybar'_\lambda=\spec(\Orb_{\Ybar'_J}/R_{\Ybar'_J}),\quad
Y_\lambda=\spec(\Orb_{Y_J}/R_{Y_J}),\quad\text{and}\quad
X'_\lambda=\spec(\Orb_{X'_J}/R_{X'_J})$$
Then $[Y \rightarrow Y' \leftarrow X']=\varprojlim_\lambda [Y_\lambda \rightarrow \Ybar'_\lambda \leftarrow X'_\lambda]$.
Finally, we take $X_\lambda=X'_\lambda\times_{\Ybar'_\lambda} Y_\lambda$ so
that $[Y \xleftarrow{f} X \xrightarrow{{i}} X']=\varprojlim_\lambda
[Y_\lambda \xleftarrow{f_\lambda} X_\lambda \xrightarrow{{i}_\lambda}
  X']$. Indeed, $X=X'\times_{Y'} Y$ and inverse limits commute with fiber
products.

For the last assertion, we note that all schemes are affine and that there
is an exact sequence
$$0\to \Gamma(\Orb_{Y'})\to \Gamma(\Orb_Y)\times \Gamma(\Orb_{X'})\to \Gamma(\Orb_X)\to 0$$
and similarly for the approximations $Y'_\lambda$ (which can be different from
$\Ybar'_\lambda$). As direct limits are exact
it follows that $Y'=\varprojlim Y'_\lambda$.
\end{proof}
\bibliography{references}
\bibliographystyle{dary}
\end{document}